\documentclass[11pt,leqno]{article}
\usepackage{mathrsfs,amssymb,amsmath,amsthm, color}
\usepackage[dvips]{graphicx}

\makeatletter
	
	\@addtoreset{equation}{section}
\makeatother

\begin{document}

\renewcommand{\theenumi}{\rm (\roman{enumi})}
\renewcommand{\labelenumi}{\rm \theenumi}

\newtheorem{thm}{Theorem}[section]
\newtheorem{defi}[thm]{Definition}
\newtheorem{lem}[thm]{Lemma}
\newtheorem{prop}[thm]{Proposition}
\newtheorem{cor}[thm]{Corollary}
\newtheorem{exam}[thm]{Example}
\newtheorem{conj}[thm]{Conjecture}
\newtheorem{rem}[thm]{Remark}
\allowdisplaybreaks

\title{Recurrence of direct products of diffusion processes in random media having zero potentials}

\author{Daehong Kim\footnote{The first named author is partially supported by a Grant-in-Aid for
Scientific Research (C)  No.~17K05304 from Japan Society for the Promotion of Science.}~ and Seiichiro Kusuoka\footnote{The second named author is partially supported by a Grant-in-Aid for Scientific Research (C) No.~17K14204 from Japan Society for the Promotion of Science.}}

\date{\empty}

\maketitle
\begin{abstract}
In this paper, we consider the recurrence of some multi-dimensional diffusion processes in random environments including zero potentials.
Previous methods on diffusion processes in random environments are not applicable to the case of such environments.
In main theorems, we obtain a sufficient condition to be recurrent for the product of a multi-dimensional diffusion process in semi-selfsimilar random environments and one-dimensional Brownian motion, and also more explicit sufficient conditions in the case of Gaussian random environments and random environments generated by L\'evy processes.
To prove them, we introduce an index which measures the strength of recurrence of symmetric Markov processes, and give some sufficient conditions for recurrence of direct products of symmetric diffusion processes.
The index is given by the Dirichlet forms of the Markov processes.
\end{abstract}
%

\section{Introduction}\label{sec:intro}
Global properties of stochastic processes as well as related problems are important topics in both probability and potential theories. Among those, recurrence and transience of Markov processes have been studied by many authors under various probabilistic and analytic aspects in discrete and in continuous time. For instance, it is well-known that a $d$-dimensional Brownian motion consisting of $d$ independent one-dimensional standard Brownian motions is recurrent if $d=1, 2$, and transient otherwise. For more general diffusion processes, we have also many criteria for their recurrence and transience, but the criteria are not always so easy to be checked. In general, whether diffusion processes are recurrent or transient depends on their generators (see \cite{FOT}, \cite{ICHI}, \cite{IW}). In this spirit, Ichihara \cite{ICHI} gave elegant criteria for the recurrence and transience of the diffusion process associated with a second order elliptic partial differential operator $L$ on $\mathbb{R}^d$ defined by
\begin{equation}\label{Elliptic}
L=\sum_{i,j=1}^d\frac{\partial}{\partial x_i}\left(a_{ij}(x)\frac{\partial}{\partial x_j}\right),
\end{equation}
where $a_{ij}(x)$ is a symmetric coefficient function such that the matrix $A(x):=(a_{ij}(x))_{1\le i,j \le d}$ is strictly positive definite on $\mathbb{R}^d$. 

Let ${\mathcal W}$ be the space of locally bounded Borel measurable functions on $\mathbb{R}^d$ vanishing at the origin and let ${\mathcal Q}$ be a probability measure on ${\mathcal W}$. In the present paper, an element of ${\mathcal W}$ is called an environment. Given an environment $w$, consider ${\bf Y}^w=(Y_w(t), P_x^w, x \in \mathbb{R}^d)$, the diffusion process with generator
\begin{equation*}
\frac{1}{2}(\Delta - \nabla w \cdot \nabla)=\frac{1}{2}e^w\sum_{k=1}^d\frac{\partial}{\partial x_k}\left(e^{-w}\frac{\partial}{\partial x_k}\right).
\end{equation*}
It is well-known that $Y_w(t)$ can be constructed from the diffusion process $X_w(t)$ associated with (\ref{Elliptic}) provided $a_{ij}=\frac{1}{2}\delta_{ij}e^{-w}$ through a random time change of $X_w(t)$. 
We call a stochastic process ${\mathcal Y}^w=(Y_w(t), {\mathcal Q}\otimes P_x^w, x \in \mathbb{R}^d)$ the diffusion process in a random environment. In the case where $d=1$ and $(w, {\mathcal Q})$ is a Brownian environment, Brox \cite{Brox} noticed that the process ${\bf Y}^w$ is a continuous version of Sinai's walk (see \cite{Sinai}) and showed that $Y_w(t)$ moves very slowly in some sense by the effect of the environment. Later, Brox's result was extended to a multi-dimensional diffusion process in a non-negative L\'evy's Brownian environment (see \cite{KIM}, \cite{Mathieu}).

Recurrence and transience of multi-dimensional diffusion processes in various random environments have been studied by many authors, in combining Ichihara's criteria with the ergodic aspects of measure preserving transformations on the random environments. 
The first result on this problem was obtained by Fukushima et al. in a one-dimensional Brownian environment (see \cite{FNT}). Tanaka considered the diffusion process ${\bf Y}^w$ in a L\'evy's Brownian environment and proved that it is to be recurrent for almost all environments in any dimension (see \cite{Tanaka93}), which made the effect of random environments on this problem quite transparent. After that, Tanaka's result was extended to a large class of multi-dimensional random environments (see \cite{KTT1}, \cite{KTT3}, \cite{Takahashi04}, \cite{TakaTamu}). In \cite{KTT3}, the authors considered multi-dimensional diffusion processes in multi-parameter random environments and studied their recurrence and transience. More precisely, the authors obtained some conditions for the dichotomy of recurrence and transience for $d$-dimensional diffusion process $Y_w(t)=(Y_{w}^1(t), Y_{w}^2(t), \cdots, Y_{w}^d(t))$ corresponding to the generator
\begin{equation}\label{directgenerator}
\frac{1}{2}\sum_{k=1}^de^{w(x_k)}\frac{\partial}{\partial x_k}\left(e^{-w(x_k)}\frac{\partial}{\partial x_k}\right) ,
\end{equation}
where $w$ is a one-dimensional (semi-)stable L\'evy process whose values at different $d$ points are regarded as constituting a multi-parameter environment. In their proof, the following property of the environments was crucial: for any $a_0 >0$ and $\theta \ge 1$
\begin{equation}\label{condition}
{\mathcal Q}\left( \inf_{\sigma \in S^{d-1}}\sum_{j=1}^dw(\theta \sigma_{j}) >a_0\right)>0,
\end{equation}
where $S^{d-1}$ denotes the unit sphere in $\mathbb{R}^d$. It turned out that the property (\ref{condition}) works well with Ichihara's test in studying the recurrence and transience of $Y_w(t)$. 
However, the property (\ref{condition}) does not hold if one component of $w$ takes value identically zero. Indeed, for the two-dimensional direct product of diffusion process $(Y_w^1(t), B(t))$ given by the pair of the Brox's diffusion $Y_w^1(t)$ and a one-dimensional Brownian motion $B(t)$ independent of $Y_w^1(t)$, let $\widetilde{w}$ be the environment relative to $(Y_w^1(t), B(t))$. Then $\widetilde{w}(x_1,x_2)=w(x_1)$ and hence $\widetilde{w}(\sigma )=w(0)=0$ for $\sigma :=(0,1) \in S^1$. 
In this sense, a $(d+1)$-dimensional diffusion process in $d$-parameter random environments 
\begin{equation}\label{directproc}
(Y_w(t),B(t)):=(Y_w^{1}(t), Y_w^{2}(t), \cdots, Y_w^{d}(t), B(t))
\end{equation}
with a one-dimensional Brownian motion $B(t)$ independent of $\{Y_w^{j}(t), j=1,2,\cdots, d\}$ is out of the framework of \cite{KTT3} (also of \cite{TakaTamu}).  

The purpose of this paper is to study the recurrence of some multi-dimensional diffusion processes in random environments including zero potentials. For this, we introduce a criterion for the recurrence of direct products of symmetric Markov processes motivated by Okura \cite{Ok}.
In the criterion, the index induced by the Dirichlet forms plays an important role as representing the strength of recurrence of the associated Markov processes.  
The criterion works well in the case of diffusion processes in random environments, and we are able to show that the diffusion processes in semi-selfsimilar random environments have very strong recurrence in sense of the index.
As a result, we can show the recurrence of direct products of Markov processes given by the pair of a $d$-dimensional diffusion process in almost all environments having usual randomness, and a one-dimensional Brownian motion (see Theorems \ref{thm:RE1}, \ref{thm:Gauss} and \ref{thm:Levy}).

\section{Main results}\label{sec:thms}
Now, we state our framework and the main results of the present paper. 

Let ${\mathcal W}$ be the space of locally bounded and Borel measurable functions on ${\mathbb R}^{d}$ with the topology generated by the uniform convergence on compact sets.
Let ${\mathscr B}({\mathcal W})$ be the Borel $\sigma$-field of ${\mathcal W}$ and ${\mathcal Q}$ be a probability measure on $({\mathcal W}, {\mathscr B}({\mathcal W}))$. We call an element $w \in {\mathcal W}$ an environment and assume that ${\mathcal Q}(w(0)=0) =1$.
For given $w\in {\mathcal W}$, let ${\bf X}^w=(X_w(t), P_x^w, x \in \mathbb{R}^d)$ be the diffusion process associated with the generator \eqref{Elliptic} provided $a_{ij}=\frac{1}{2}\delta_{ij}e^{-w}$, equivalently, associated with the strongly local Dirichlet form $({\mathcal E}^{w}, {\mathcal F}^w)$ defined by
\begin{equation}\label{eq:DFw}
\left\{
\begin{array}{ll}
\displaystyle {\mathcal F}^w  := \left\{ f\in L^2({\mathbb R}^d; e^{-w}{\rm d}x ) : \frac{\partial f}{\partial x^i} \in L^2({\mathbb R}^d; e^{-w}{\rm d}x), i=1,2,\dots ,d \right\} ,\\%
\displaystyle {\mathcal E}^w (f,g) \displaystyle := \frac 12 \int _{{\mathbb R}^d} \nabla f (x) \cdot \nabla g(x)\, e^{-w(x)} {\rm d}x, \quad f,g \in {\mathcal F}^w,
\end{array}
\right.
\end{equation}
where the derivatives ${\partial f}/{\partial x_i}$ are taken in the sense of Schwartz distributions. 

For $r >1$ and $\alpha >0$, let $T$ be a mapping from ${\mathcal W}$ to ${\mathcal W}$ defined by $Tw(x) = r^{-\alpha} w(rx)$ for $x\in {\mathbb R}^d$. We assume that
\begin{equation}\label{eq:selfsimilar}
{\mathcal Q}(A) = {\mathcal Q}(TA) ,\quad A\in {\mathscr B}({\mathcal W}),
\end{equation}
which implies that $T$ is a measure preserving transformation of ${\mathcal Q}$.
We call a space $({\mathcal W}, {\mathscr B}({\mathcal W}), {\mathcal Q})$ satisfying the condition (\ref{eq:selfsimilar}) an $\alpha$-semi-selfsimilar random environment.

Our first result concerns a sufficient condition to be recurrent for the direct product of a $d$-dimensional diffusion process in semi-selfsimilar random environments and a one-dimensional Brownian motion.

\begin{thm}\label{thm:RE1}
Assume that $T$ is weakly mixing and
\begin{align}\label{assump;RE1}
{\mathcal Q} \left(\overline{w}(1, r^2) -\underline{w}(0,1) <2\underline{w}(1, r^2) \right) >0,
\end{align}
where $\overline{w}$ and $\underline{w}$ are given by
$$
\overline{w}(a,b):=\sup_{a \le |x| \le b}w(x), \quad \underline{w}(a,b):=\inf_{a \le |x| \le b}w(x)
$$
for $a, b \ge 0$.
Then, the $(d+1)$-dimensional direct product of diffusion process $(X_{w}(t), B(t))$ given by the pair of the $d$-dimensional diffusion process $X_{w}(t)$ and a one-dimensional Brownian motion $B(t)$ independent of $X_{w}(t)$ is recurrent for almost all environments.
\end{thm}

It is known that random environments make a strong effect for the recurrence of the diffusion process in random environments.
Indeed, diffusion processes in various random environments are recurrent in any dimension under some natural conditions (cf. \cite{KTT1}, \cite{KTT3}, \cite{Takahashi04}, \cite{TakaTamu}). 
However, we cannot apply the previous methods to prove Theorem \ref{thm:RE1}, because the component $B(t)$ of the product process $(X_{w}(t), B(t))$ has no effect of environments. As an alternative to the previous method, 
we employ a new criterion for the recurrence of direct products of symmetric Markov processes based on the theory of Dirichlet forms (Proposition \ref{thm:criterion} below). The criterion will be applied to diffusion processes in random environments together with Proposition \ref{prop:RE}, which plays a key role for the proof of Theorem \ref{thm:RE1}. 

We remark that
the $(d+d')$-dimensional direct product process $(X_{w}(t), B^1(t), \dots , B^{d'}(t) )$ given by the pair of  $X_w(t)$
and a $d'$-dimensional Brownian motion $(B^1(t), \dots , B^{d'}(t) )$ independent of $X_{w}(t)$, is transient whenever $d'\ge 3$, because the marginal $(B^1(t), \dots , B^{d'}(t) )$ is transient.
To our best knowledge, the case $d'=2$ is an open problem but we believe that it will be transient in view of the result in  discrete cases. 
In fact, a similar problem was already concerned in discrete cases. In \cite{Devulderetal}, the authors considered $d+d'$ independent walks on $\mathbb{Z}$, $d$ of them performing Sinai's walk and $d'$ of them performing simple symmetric random walk, and proved that the direct product is recurrent almost all environments if and only if $d'\le 1$, or $d=0$ and $d'=2$. 

Next, we consider specific laws of random environments.
Precisely, we show some sufficient conditions in the cases that Gaussian random environments and random environments generated by L\'evy processes. Such environments are concerned in \cite{KTT1} and \cite{KTT3}, respectively.
By giving some assumptions on laws of environments we are able to discuss clearer sufficient conditions for the recurrence as follows, while the sufficient condition given in Theorem \ref{thm:RE1} is somewhat abstract.

Let us consider a probability measure ${\mathcal Q}$ on $({\mathcal W}, {\mathscr B}({\mathcal W}))$ as a Gaussian measure. 
We assume that ${\mathcal Q}(w(0)=0)=1$ and $E^{\cal Q}[w(x)]=0$ for $x\in {\mathbb R}^d$. Here $E^{\cal Q}$ stands for the expectation with respect to ${\cal Q}$. 
Let $K$ be the covariance kernel of ${\mathcal Q}$, that is, $K(x,y) = E^{\cal Q}[w(x)w(y)]$ for $x,y\in {\mathbb R}^d$. Since ${\mathcal Q}$ is a probability measure on $({\mathcal W}, {\mathscr B}({\mathcal W}))$, $K$ is a measurable function on ${\mathbb R}^d\times {\mathbb R}^d$.

\begin{thm}\label{thm:Gauss}
Assume that
\begin{align}
\label{eq:thmGauss1} &\sup _{1\le |x| \le r^2} \int _{|y|\leq r^2} K(x,y) {\rm d}y -\inf _{|x|\le 1} \int _{|y|\leq r^2} K(x,y) {\rm d}y <2\inf _{1\le |x|\le r^2} \int _{|y|\leq r^2} K(x,y) {\rm d}y ,\\
\label{eq:thmGauss2}& \lim _{n\rightarrow \infty} r^{-\alpha n} \sup _{x,y\in D_1} K(r^n x,y) =0.
\end{align}
Then, the $(d+1)$-dimensional direct product of diffusion process $(X_{w}(t), B(t))$ given by the pair of $X_{w}(t)$ and a one-dimensional Brownian motion $B(t)$ independent of $X_{w}(t)$ is recurrent for almost all environments.
\end{thm}

Note that the condition \eqref{eq:thmGauss1} above is stronger than those in Theorem 3.5 in \cite{KTT1}, because we need the strength of the recurrence for $X_w(t)$ to show the recurrence of the direct product process.

As a direct consequence of Theorem \ref{thm:Gauss}, we have the following corollary.

\begin{cor}\label{cor:BroxBrown}
The two-dimensional direct product of diffusion process $(Y_w(t), B(t))$ given by the pair of the Brox's diffusion process $Y_w(t)$ and a one-dimensional Brownian motion $B(t)$ independent of $Y_w(t)$ is recurrent for almost all environments.
\end{cor}

A similar problem was concerned in Proposition 3.1 in \cite{KTT3}, but the absolute value of the Brownian environment was taken in the component of the diffusion process in random environment.
The advantage of the method in the present paper is that we are able to show the recurrence without taking the absolute value.
Proofs of Theorem \ref{thm:Gauss} and Corollary \ref{cor:BroxBrown} are given in Section \ref{subsec:Gauss}.

Finally, we give the result on the case that random environments generated by L\'evy processes.
Let ${\mathcal W}$ be the space of functions $w$ on ${\mathbb R}$ satisfying the following: $w(0)=0$, $w$ is right (resp. left) continuous with left (resp. right) limits on $[0,\infty)$ (resp. $(-\infty, 0)$).
For $i=1,2,\dots , d$, we set a probability measure ${\mathcal Q}_i$ on $({\mathcal W}, {\mathscr B}({\mathcal W}))$ such that $(w(x), x \in [0,\infty), {\mathcal Q}_i)$ and $(w(-x), x \in  [0,\infty ), {\mathcal Q}_i)$ are independent semi-selfsimilar L\'evy processes with an exponent $\alpha _i \in (0,2]$ (see Definitions 13.4 and 13.12 in \cite{Sa}). 
Define the probability measure ${\mathcal Q}$ on $({\mathcal W}^d, {\mathscr B}({\mathcal W})^{\otimes d})$ by
\[
{\mathcal Q} := \bigotimes _{i=1}^d {\mathcal Q_i} .
\]
Denote the $i\,$th component of $w \in {\mathcal W}^d$ by $w^i$ and write $w_+^i (t) := w^i(t)$ and $w_-^i (t) := w^i(-t)$ for $t\in [0,\infty )$.

\begin{thm}\label{thm:Levy}
Let $i=1,2,\dots, d$. If $\alpha _i =2$ or both $w_+^i$ and $w_-^i$ have positive jumps with positive probabilities, then the $(d+1)$-dimensional direct product process  $(X_{w}(t), B(t))$ given by the pair of $X_{w}(t)$ and a one-dimensional Brownian motion $B(t)$ independent of $X_{w}(t)$ is recurrent for almost all environments.
\end{thm}

This result can be regarded as an extension of Theorem 1.2 (i) in \cite{KTT3}.
Theorem \ref{thm:Levy} implies that diffusion processes in random environments generated by L\'evy processes with positive jumps have very strong recurrence, and even taking a direct product of the diffusion process and a one-dimensional Brownian motion, the direct product process is still recurrent.
Proof of Theorem \ref{thm:Levy} is given in Section \ref{subsec:Levy}.
\vskip 0.2cm
The rest of the present paper is organized as follows. In Section \ref{sec:prod}, we give criteria for the recurrence of direct products of general symmetric Markov processes including a random time changed version, and prove some lemmas on diffusion processes in non-random environments. In Section \ref{sec:RE}, we give some sufficient conditions on the random environment for the recurrence of a multi-dimensional direct product process in an ergodic random environment. In Section \ref{sec:App}, we consider concrete examples for the result obtained in Section \ref{sec:RE} with Gaussian and stable L\'evy environments. 
For notational convenience, we let $a \wedge b:=\min\{a, b\}$ for any $a, b \in \mathbb{R}$.

\section{Recurrence of products of Dirichlet forms and some lemmas on diffusion processes in non-random environments}\label{sec:prod}

\subsection{Recurrence of products of Dirichlet forms}\label{sec:prod-1}
In this section, we give some analytic recurrence criteria for direct products of symmetric Dirichlet forms (or, of symmetric Markov processes). The result will be obtained by a simple obsevation for the recurrence of direct products of symmetric Markov processes due to \cite{FuO} and \cite{Ok}. 

For $i=1,2,\dots, N$, let $E^{(i)}$ be a locally compact separable metric space and $m^{(i)}$ be a positive Radon measure on $E^{(i)}$ with full support. Let $({\mathcal E}^{(i)}, {\mathcal F}^{(i)})$ be a symmetric regular Dirichlet form on $L^2(E^{(i)}, m^{(i)})$ possessing ${\mathcal C}^{(i)}$ as its core.  
It is well-known that $({\mathcal E}^{(i)}, {\mathcal F}^{(i)})$ generates a strongly continuous Markovian semigroup $(T_t^{(i)})_{t\ge 0}$ of symmetric operators on $L^2(E^{(i)},m^{(i)})$. 
Let ${\bf X}^{(i)}= (\Omega ^{(i)}, {\mathcal M}^{(i)}, X^{(i)}(t), P_{x^{(i)}}^{(i)})$ be the $m^{(i)}$-symmetric Hunt process associated to $({\mathcal E}^{(i)}, {\mathcal F}^{(i)})$. We say that $({\mathcal E}^{(i)}, {\mathcal F}^{(i)})$ (or ${\bf X}^{(i)}$) is irreducible if any $(T_t^{(i)})_{t\ge 0}$-invariant set $B$ satisfiess $m^{(i)}(B)=0$ or $m^{(i)}(E^{(i)}\setminus B)=0$. 

Let ${\bf X}= (\Omega , {\mathcal M}, X(t), P_x)$ be the process on $E$ defined by the product of ${\bf X}^{(i)}$, where 
\begin{align*}
&\Omega = \Omega ^{(1)} \times \Omega ^{(2)}\times \cdots \times \Omega ^{(N)}, \quad {\mathcal M} = {\mathcal M}^{(1)} \otimes {\mathcal M}^{(2)} \otimes \cdots \otimes {\mathcal M}^{(N)},\\
& E = E^{(1)} \times E^{(2)} \times \cdots \times E^{(N)},\\
&P_x = P_{x^{(1)}}^{(1)} \otimes P_{x^{(2)}}^{(2)} \otimes \cdots P_{x^{(N)}}^{(N)}, \quad x= (x^{(1)}, x^{(2)}, \dots , x^{(N)}) \in E,\\
&X(t,\omega ) = \left( X^{(1)}(t,\omega _1), X^{(2)}(t,\omega _2), \dots , X^{(N)}(t,\omega _N) \right) , \quad \omega =(\omega _1, \omega _2, \dots , \omega_N )\in {\Omega}.
\end{align*}
We note that the marginal processes $\{(X^{(i)}(t), t\ge 0), i=1,2, \dots, N\}$ are independent under $P_x$.
Let $m$ be the product measure of $\{ m^{(i)}, i=1,2, \dots, N\}$. Assume that ${\bf X}^{(i)}$ is irreducible for any $i=1,2,\dots, N$. 
Then, ${\bf X}$ is also to be an $m$-symmetric irreducible Markov process on $E$ (\cite[Proposition 3.1]{FuO}, \cite[Theorem 2.6]{Ok}). Let $({\mathcal E}, {\mathcal F})$ be the associated Dirichlet form of ${\bf X}$ on $L^2(E,m)$. Then $({\mathcal E}, {\mathcal F})$ possesses the linear span of ${\mathcal C}^{(1)} \otimes \dots \otimes {\mathcal C}^{(N)}:=\{\phi^{(1)}\otimes \cdots \otimes \phi^{(N)} : \phi^{(i)} \in {\mathcal C}^{(i)}, i=1,2,\dots, N\}$ as its core, where $(\phi^{(1)}\otimes \cdots \otimes \phi^{(N)})(x) := \phi ^{(1)} (x^{(1)}) \phi ^{(2)} (x^{(2)}) \dots \phi ^{(N)}(x^{(N)})$. Thus the Dirichlet form $({\mathcal E}, {\mathcal F})$ is to be regular and also admits the following expressions: for $u^{(i)} \in {\mathcal F}^{(i)}$ ($i=1,2,\dots, N$), $u:=u^{(1)}\otimes \cdots \otimes u^{(N)} \in {\mathcal F}$ and
\begin{equation}\label{eq:Dprod}
{\mathcal E}\left(u,u\right) = \sum_{i=1}^N {\mathcal E}^{(i)} \left(u^{(i)},u^{(i)}\right) \prod _{j=1, j\neq i}^N \left(u^{(j)},u^{(j)} \right) _{m^{(j)}}
\end{equation}
where $(\cdot , \cdot )_{m^{(i)}}$ denotes the inner product on $L^2(E^{(i)}, m^{(i)})$ (\cite[Theorems 1.3 and 1.4]{Ok}). 

Let $X$ be a locally compact separable metric space and $\mu$ a positive Radon measure on $X$ with full support. A regular Dirichlet form $({\mathcal A}, {\mathcal V})$ on $L^2(X,\mu)$ (or the corresponding Markov process ${\bf M}$) is non-transient if and only if the following property holds:
\begin{enumerate}
\item[(R)] There exists a sequence $\{u_n\}_{n\ge 1} \subset {\mathcal V}$ such that $0 \le u_n \le 1$ $\mu$-a.e., $\lim_{n\to \infty}u_n=1$ $\mu$-a.e. and ${\mathcal A}(u_n,u_n) \to 0$ as $n\to \infty$ 
\end{enumerate}
(see \cite[Theorem 1.6.3]{FOT}). In particular, it is known that $({\mathcal A}, {\mathcal V})$ (or ${\bf M}$) is to be recurrent if it is irreducible and possesses the property (R) (\cite[Lemma 1.6.4]{FOT}).
\vskip 0.2cm
Now we give some simple criteria for the non-transience of ${\bf X}$ through the marginal processes $\{ {\bf X}^{(i)}, i=1,2,\dots, N\}$ in an analytic way. 

\begin{prop}\label{thm:criterion}
Let $\{({\mathcal E}^{(i)}, {\mathcal F}^{(i)}), i=1,2,\cdots, N\}$ and $({\mathcal E}, {\mathcal F})$ be as above.
Assume that there exist sequences $\{u^{(i)}_n\}_{n\ge 1} \subset {\mathcal F}^{(i)}$ such that~ $0 \le u^{(i)}_{n^{(i)}(k)} \le 1$ $m^{(i)}$-a.e., $\lim_{k\to \infty}u^{(i)}_{n^{(i)}(k)}=1$ $m^{(i)}$-a.e., $u^{(i)}_{\infty}:=1$ $m^{(i)}$-a.e. $(i=1,2, \dots, N)$ and 
\begin{equation}\label{1kinner}
\lim _{k\rightarrow \infty} \frac{1}{k}\, \prod _{i=1}^N \left(u_{n^{(i)}(k)}^{(i)}, u_{n^{(i)}(k)}^{(i)} \right) _{m^{(i)}} =0
\end{equation}
for the index
\begin{align}\label{ID}
n^{(i)}(k) := \inf \left\{ n\in {\mathbb N} :  {\mathcal E}^{(i)}\left(u_n^{(i)}, u_n^{(i)}\right) \leq \frac 1k \left(u_n^{(i)}, u_n^{(i)}\right)_{m^{(i)}} \right\} , \quad k\in {\mathbb N}. 
\end{align}
Then the Dirichlet form $({\mathcal E}, {\mathcal F})$ $($or the direct product process ${\bf X}$ of $\{ {\bf X}^{(i)}, i=1,2,\dots, N\}$$)$ 
is non-transient.
\end{prop}

\begin{proof}
Let 
$$
u_{n(k)}:=u_{n^{(1,\dots, N)}(k)}=u_{n^{(1)}(k)}^{(1)}\otimes \cdots \otimes u_{n^{(N)}(k)}^{(N)}.
$$
It is then easy to see that $0 \le u_{n(k)} \le 1$ $m$-a.e. and $u_{n(k)} \rightarrow 1$ as $k \to \infty$ $m$-a.e. Moreover, by (\ref{eq:Dprod}) and the assumption (\ref{1kinner}), we have
\begin{align*}
{\mathcal E}\left(u_{n(k)},u_{n(k)}\right)&= \sum_{i=1}^N {\mathcal E}^{(i)} \left(u^{(i)}_{n^{(i)}(k)},u^{(i)}_{n^{(i)}(k)}\right) \prod _{j=1, j\neq i}^N \left(u^{(j)}_{n^{(j)}(k)},u^{(j)}_{n^{(j)}(k)} \right) _{m^{(j)}} \\
&\le \frac{N}{k}\,\prod _{i=1}^N \left(u_{n^{(i)}(k)}^{(i)}, u_{n^{(i)}(k)}^{(i)} \right)_{m^{(i)}} \longrightarrow 0 \quad {\rm as}~~k \to \infty.
\end{align*}
Hence $({\mathcal E}, {\mathcal F})$ (or ${\bf X}$) is non-transient.  
\end{proof}

For a strictly positive continuous function $g^{(i)}$ on $E^{(i)}$, let $\{{\bf Y}^{(i)}, i=1,2,\dots, N\}$ be the time changed processes of $\{{\bf X}^{(i)}, i=1,2,\dots, N\}$ defined by 
$$
Y^{(i)}(t):=X^{(i)}\left(\tau_t^{(i)}\right),
$$
where $\tau^{(i)}_t$ is the right continuous inverse of the positive continuous additive functional $A_t^{(i)}=\int_0^tg^{(i)}(X^{(i)}(s)){\rm d}s$ of ${\bf X}^{(i)}$, that is, $\tau^{(i)}_t=\inf\{s>0 : A_s^{(i)} > t\}$ $(i=1,2,\dots, N)$. Then, since the fine support of $g^{(i)}m^{(i)}$ equals to $E^{(i)}$, $\{{\bf Y}^{(i)}, i=1,2,\dots, N\}$ are $g^{(i)}m^{(i)}$-symmetric Humt processes on $E^{(i)}$. Note that the irreducibility and non-transience are stable under time-changed transform (see \cite[Theorem 6.2.3]{FOT} and \cite[Theorems 8.2 and 8.5]{Sil}). Hence $\{{\bf Y}^{(i)}, i=1,2,\dots, N\}$ are irreducible and non-transient if $\{{\bf X}^{(i)}, i=1,2,\dots, N\}$ are so. Let $\{(\check{\mathcal E}^{(i)}, \check{\mathcal F}^{(i)}), i=1,2,\dots N\}$ be the associated Dirichlet forms of $\{{\bf Y}^{(i)}, i=1,2,\dots, N\}$ on $L^2(E^{(i)}, g^{(i)}m^{(i)})$. Then $(\check{\mathcal E}^{(i)}, \check{\mathcal F}^{(i)})$ is given by
\begin{align*}
\left\{\begin{array}{ll}
\check{\mathcal F}^{(i)}=\left\{\varphi\in L^2\left(E^{(i)}, g^{(i)}m^{(i)}\right) : \text{there exists }f\in {\mathcal F}^{(i)}_e\text{ such that }\varphi=f\ g^{(i)}m^{(i)}\text{-a.e.}\right\} \\ 
\check{\mathcal E}^{(i)}(\varphi,\varphi)={\mathcal E}^{(i)}(f,f)\quad\text{ for }\varphi \in \check{\mathcal F}^{(i)}\text{ and } f\in {\mathcal F}^{(i)}_e\text{ with }\varphi=f\ g^{(i)}m^{(i)}\text{-a.e.,}
\end{array}\right.
\end{align*}
where ${\mathcal F}^{(i)}_e$ is the extended Dirichlet space of ${\mathcal F}^{(i)}$ (see \cite{FOT} for the definition). 
Then we can obtain the following corollary as a consequence of Proposition \ref{thm:criterion}. 

\begin{cor}\label{maincor}
For $i=1,2\dots, N$, let $g^{(i)}$ be a strictly positive continuous function on $E^{(i)}$. 
Assume that the marginal processes $\{{\bf X}^{(i)}, i=1,2,\dots, N\}$ are irreducible and non-transient.
If there exist sequences $\{u^{(i)}_n\}_{n\ge 1} \subset {\mathcal F}^{(i)}_e \cap L^2(E^{(i)}, g^{(i)}m^{(i)})$ such that~ $0 \le u^{(i)}_{n^{(i)}(k)} \le 1$ $g^{(i)}m^{(i)}$-a.e., $u^{(i)}_{n^{(i)}(k)} \rightarrow 1$ as $k \to \infty$ $g^{(i)}m^{(i)}$-a.e. $(i=1,2, \dots, N)$ and
\begin{equation}\label{1kinner2}
\lim _{k\rightarrow \infty} \frac{1}{k}\, \prod _{i=1}^N \left(u_{n^{(i)}(k)}^{(i)}, u_{n^{(i)}(k)}^{(i)} \right) _{g^{(i)}m^{(i)}} =0,
\end{equation}
for the index
\begin{equation}\label{nik}
n^{(i)}(k) := \inf \left\{ n\in {\mathbb N} :  {\mathcal E}^{(i)}\left(u_n^{(i)}, u_n^{(i)}\right) \leq \frac 1k \left(u_n^{(i)}, u_n^{(i)}\right)_{g^{(i)}m^{(i)}} \right\} , 
\end{equation}
then the direct product process ${\bf Y}$ of $\{ {\bf Y}^{(i)}, i=1,2,\dots, N\}$ is recurrent.
\end{cor}

\subsection{Some lemmas on diffusion processes in non-random environments}\label{sec:NRE}

Let $w$ be a locally bounded and Borel measurable function on ${\mathbb R}^d$.
Consider the strongly local Dirichlet form $({\mathcal E}^{w}, {\mathcal F}^w)$ defined by (\ref{eq:DFw}).
Denote  $C_0^\infty ({\mathbb R}^d)$ by the set of all smooth functions with compact support in $\mathbb{R}^d$.
Note that the local boundedness of $w$ implies that $C_0^\infty ({\mathbb R}^d)$ is dense in ${\mathcal F}^w$, in particular $({\mathcal E}^{w}, {\mathcal F}^w)$ is regular. 
Let ${\bf X}^{w}= (X_{w}(t), P_{x}^{w})$ be the diffusion process associated with $({\mathcal E}^{w}, {\mathcal F}^{w})$. The $d$-dimensional Brownian motion is associated to $({\mathcal E}^{0}, {\mathcal F}^{0})$, the Dirichlet form $({\mathcal E}^{w}, {\mathcal F}^{w})$ with $w\equiv 0$. 

For $r\in (1,\infty )$, let $\varphi \in C_0^\infty ({\mathbb R}^d)$ such that $0\leq \varphi(x) \leq 1$ on $\mathbb{R}^d$, $\varphi (x) =1$ on $|x|\leq 1$, and $\varphi (x) =0$ on $|x| \geq r$.
For fixed $r$ and $\varphi$, define the sequence $\{ u_n\} \subset {\mathcal F}^{w}$ by
\begin{equation}\label{eq:defun}
u_n(x) := \varphi \left( r^{-n} x \right) ,\quad x\in {\mathbb R}^d, \ n\in {\mathbb N} \cup \{ 0\}.
\end{equation}
It is clear that $\lim _{n\rightarrow \infty} u_n(x) =1$ for $x\in {\mathbb R}^d$. 
We let 
$$
\overline{w}(a,b):=\sup_{a \le |x| \le b}w(x), \quad \underline{w}(a,b):=\inf_{a \le |x| \le b}w(x)
$$
for $a, b \ge 0$ as in Theorem \ref{thm:RE1}.
Then, it is easy to see by the definition of $u_n$ and the assumption on $\varphi$ that for $n\in {\mathbb N}$,
\begin{equation}\label{lem:w1-2} 
C_1 r^{dn} \exp \left(-\overline{w}(0,r^{n+1})\right)
\leq \int _{{\mathbb R}^d} |u_n(x)|^2 e^{-w(x)}{\rm d}x \leq C_1 r^{dn} \exp \left(-\underline{w}(0,r^{n+1})\right),
\end{equation}
where $C_1:=C_1(d, r,\varphi) = \int _{|x|\le r} \varphi \left( x\right) ^2 {\rm d}x$. 
Moreover, since 
$$
\int _{{\mathbb R}^d} |u_n(x)|^2 e^{-w(x)}{\rm d}x \geq r^{dn} \int _{|x|\leq r^{-\ell}} e^{-w(r^n x)}{\rm d}x
$$
for any $\ell \in \mathbb{N}$, it also follows that for $n \in \mathbb{N}$
\begin{equation}\label{lem:w1-3}
\int _{{\mathbb R}^d} |u_n(x)|^2 e^{-w(x)}{\rm d}x \geq V_d\, r^{d(n-\ell)} \exp \left(-\overline{w}(0,r^{n-\ell})\right).
\end{equation}
In particular
\begin{equation}\label{lem:w1-4}
\int _{{\mathbb R}^d} |u_n(x)|^2 e^{-w(x)}{\rm d}x \ge V_d \exp \left(-\overline{w}(0, 1)  \right),
\end{equation}
where $V_d$ denotes the volume of the unit ball in ${\mathbb R}^d$. On the other hand, the relation 
$$
\int _{{\mathbb R}^d} |\nabla u_n(x)|^2 e^{-w(x)}{\rm d}x = r^{(d-2)n} \int _{1\le |x| \le r} |\nabla \varphi  \left( x\right) | ^2 e^{-w(r^n x)}{\rm d}x
$$
implies that for $n\in {\mathbb N}$,
\begin{equation}\label{lem:w2}
C_2 r^{(d-2)n}\exp \left( - \overline{w}(r^n, r^{n+1}) \right)
\leq \int _{{\mathbb R}^d} |\nabla u_n(x)|^2 e^{-w(x)}{\rm d}x \leq C_2 r^{(d-2)n} \exp \left( - \underline{w}(r^n, r^{n+1})\right),
\end{equation}
where $C_2:=C_2(d, r,\varphi) = \int _{1\le |x| \le r} | \nabla \varphi  \left( x \right) |^2 {\rm d}x$.
\vskip 0.2cm

Define a number $n(k)\in {\mathbb N} \cup \{ \infty\}$ by
$$
n(k) := \inf \left\{ n\in {\mathbb N} : {\mathcal E}^w(u_n, u_n) \leq \frac{1}{k}\int_{\mathbb{R}^d}|u_n(x)|^2e^{-w(x)}{\rm d}x \right\} , \quad k\in {\mathbb N}.
$$

\begin{lem}\label{lem:w3}
Let $k\in {\mathbb N}$ such that $n(k)<\infty$ and ${\mathcal E}^w(u_{n(k)}, u_{n(k)})\neq 0$. Then we have
\begin{align*}
&\int _{{\mathbb R}^d} |u_{n(k)}(x)|^2 e^{-w(x)}{\rm d}x < \frac{C_2}{2} k r^{(d-2)(n(k)-1)} \exp \left(-\underline{w}(r^{n(k)-1}, r^{n(k)})  \right)  \\
&\quad \hspace{4cm} \times \left( 1+ \frac{(r^d-1) r^{dn(k)}\exp \left(-\underline{w}(r^{n(k)-1}, r^{n(k)+1}) \right)}{\exp \left(-\overline{w}(0, 1)  \right)} \right).
\end{align*}
\end{lem}

\begin{proof}
We note that the choice of $k \in {\mathbb N}$ and the definition of $n(k)$ imply 
\[
\int _{{\mathbb R}^d} |u_{n(k)-1}(x)|^2 e^{-w(x)}{\rm d}x < \frac{k}{2}\int _{{\mathbb R}^d} |\nabla u_{n(k)-1}(x)|^2 e^{-w(x)}{\rm d}x.
\]
In view of (\ref{lem:w1-2}), (\ref{lem:w1-4}) and (\ref{lem:w2}), we then have
\begin{align*}
&\int _{{\mathbb R}^d} |u_{n(k)}(x)|^2 e^{-w(x)}{\rm d}x 
= \frac{\int _{{\mathbb R}^d} |u_{n(k)}(x)|^2 e^{-w(x)} {\rm d}x}{\int _{{\mathbb R}^d} |u_{n(k)-1}(x)|^2 e^{-w(x)} {\rm d}x} \int _{{\mathbb R}^d} |u_{n(k)-1}(x)|^2 e^{-w(x)} {\rm d}x\\
& \quad < \frac 12 \left( 1+ \frac{\int _{{\mathbb R}^d} \left( |u_{n(k)}|^2 - |u_{n(k)-1}|^2\right) e^{-w(x)} {\rm d}x}{\int _{{\mathbb R}^d} |u_{n(k)-1}|^2 e^{-w(x)} {\rm d}x} \right) k  \int _{{\mathbb R}^d} |\nabla u_{n(k)-1}(x)|^2 e^{-w(x)}{\rm d}x\\
&\quad \leq \frac{C_2}{2} \left( 1+ \frac{\int _{r^{n(k)-1} \le |x| \le r^{n(k)+1}} e^{-w(x)} {\rm d}x}{V_d \exp \left(-\overline{w}(0, 1)  \right)} \right) k r^{(d-2)(n(k)-1)} \exp \left(-\underline{w}(r^{n(k)-1}, r^{n(k)}) \right) \\
&\quad \leq \frac{C_2}{2} \left( 1+ \frac{(r^d-1) r^{dn(k)}\exp \left(-\underline{w}(r^{n(k)-1}, r^{n(k)+1}) \right)}{\exp \left(-\overline{w}(0, 1)  \right)} \right) \\
&\quad \hspace{6cm} \times k r^{(d-2)(n(k)-1)} \exp \left(-\underline{w}(r^{n(k)-1}, r^{n(k)}) \right).
\end{align*}
\end{proof}

\begin{lem}\label{lem:w1}
Let $k\in {\mathbb N}$ such that $n(k)<\infty$ and ${\mathcal E}^w(u_{n(k)}, u_{n(k)})\neq 0$. Then we have
\begin{align*}
&\int _{{\mathbb R}^d} |u_{n(k)}(x)|^2 e^{-w(x)}{\rm d}x\\
&\qquad < \widetilde C k^{d/2}\exp \left(-\underline{w}(0, r^{n(k)+1}) - \frac{d}{2}\, \underline{w}(r^{n(k)-1}, r^{n(k)}) + \frac{d}{2}\, \overline{w}(0, r^{n(k)-2})\right),
\end{align*}
where $\widetilde C := (2V_d)^{-d/2}C_1C_2^{d/2}r^{d(d+2)/2}$.
\end{lem}

\begin{proof}
In view of (\ref{lem:w1-3}) for $\ell=1$ and (\ref{lem:w2}), we see that for $n\in {\mathbb N}$
$$
\frac{\int _{{\mathbb R}^d} |\nabla u_n(x)|^2 e^{-w(x)}{\rm d}x}{\int _{{\mathbb R}^d} |u_n(x)|^2 e^{-w(x)}{\rm d}x}  \leq \frac{C_2}{V_d}r^{-2n+d} \exp \left(- \underline{w}(r^n, r^{n+1}) + \overline{w}(0, r^{n-1})\right).
$$
From this inequality and the definition of $n(k)$, we have
$$
\frac 1k < \frac{C_2}{2V_d}r^{-2n(k)+d+2}\exp \left(-\underline{w}(r^{n(k)-1}, r^{n(k)}) + \overline{w}(0, r^{n(k)-2})\right)
$$
for $k\in {\mathbb N}$ such that $n(k)<\infty$ and ${\mathcal E}^w(u_{n(k)}, u_{n(k)})\neq 0$.
Hence
$$
r^{dn(k)}=\left(r^{2n(k)}\right)^{d/2} < \left(\frac{r^{d+2}C_2}{2V_d}\right)^{d/2}k^{d/2}\exp \left(-\frac{d}{2}\,\underline{w}(r^{n(k)-1}, r^{n(k)}) + \frac{d}{2}\,\overline{w}(0, r^{n(k)-2})\right).
$$
Applying this inequality to the upper estimate in (\ref{lem:w1-2}), we can obtain the assertion.
\end{proof}

The condition $n(k) < \infty$ for any $k \in \mathbb{N}$ is guaranteed in the case of $d$-dimensional Brownian motion (or the Dirichlet form $({\mathcal E}^0, {\mathcal F}^0)$). Therefore, by virtue of Lemma \ref{lem:w1}, we have the following fact.

\begin{cor}\label{ex:Bm} For any $k \in \mathbb{N}$, it holds that
$$
k^{-d/2}\int _{{\mathbb R}^d} |u_{n(k)}(x)|^2 {\rm d}x < \widetilde C,
$$
where $\widetilde{C}$ is the constant which appeared in Lemma \ref{lem:w1}. 
\end{cor}

\section{Recurrence of diffusion processes in random environments}\label{sec:RE}

Let ${\mathcal W}$ be the space of locally bounded and Borel measurable functions on ${\mathbb R}^{d}$ with the topology generated by the uniform convergence on compact sets.
Let ${\mathscr B}({\mathcal W})$ be the Borel $\sigma$-field of ${\mathcal W}$ and ${\mathcal Q}$ be a probability measure on $({\mathcal W}, {\mathscr B}({\mathcal W}))$. We call an element $w \in {\mathcal W}$ an environment and assume that ${\mathcal Q}(w(0)=0) =1$.
For given $w\in {\mathcal W}$, we define the Dirichlet form $({\mathcal E}^{w}, {\mathcal F}^{w})$ by (\ref{eq:DFw}) and let ${\bf X}^{w}= (X_{w}(t), P_{x}^{w})$ be the associated diffusion process of $({\mathcal E}^{w}, {\mathcal F}^{w})$.

For $r >1$ and $\alpha >0$, let $T$ be a mapping from ${\mathcal W}$ to ${\mathcal W}$ defined by $Tw(x) = r^{-\alpha} w(rx)$ for $x\in {\mathbb R}^d$. We assume the $\alpha$-semi-selfsimilarity of $({\mathcal W}, {\mathscr B}({\mathcal W}), {\mathcal Q})$ by (\ref{eq:selfsimilar}).
We say that a mapping $T$ is weakly mixing if 
$$
\lim _{n\rightarrow \infty} \frac{1}{n} \sum _{k=0}^{n-1} \left| {\mathcal Q}\left(T^k A_1 \cap A_2\right) - {\mathcal Q}(A_1){\mathcal Q}(A_2)\right| = 0, \quad A_1,A_2\in {\mathscr B}({\mathcal W}).
$$
\vskip 0.2cm 
As in the proof of Theorem 2.2 in \cite{KTT1}, we can prove the following lemma.

\begin{lem}\label{prop:mixing} Assume that $T$ is weakly mixing. If $A\in {\mathscr B}({\mathcal W})$ satisfies ${\mathcal Q}(A)>0$, then, for ${\mathcal Q}$-almost every $w\in {\mathcal W}$, $\{ n\in {\mathbb N} : T^n w \in A\}$ is an infinite set.
\end{lem}

Let $\varphi$ and $\{ u_n \}$ be the functions defined as in Section \ref{sec:prod}.
For given ${\mathbb N}$-valued increasing sequence $\{n_{\ell}\}_{\ell \ge 1}$, define $\ell(k) \in {\mathbb N} \cup \{ \infty\}$ by
$$
\ell(k):=\ell (k,w) = \inf \left\{ \ell \in {\mathbb N} : {\mathcal E}^w\!\left(u_{n_\ell}, u_{n_\ell}\right) \leq \frac{1}{k}\int_{\mathbb{R}^d}\left|u_{n_{\ell}}(x)\right|^2e^{-w(x)}{\rm d}x \right\} , \quad k\in {\mathbb N}.
$$
\vskip 0.2cm
\noindent
By using Lemma \ref{prop:mixing} above, we have the estimate as follows.

\begin{prop}\label{prop:RE}
Assume that $T$ is a weakly mixing and
\begin{equation}\label{mainassump}
{\mathcal Q} \left(\underline{w}(1, r^2) >a,~ \overline{w}(1, r^2) - \underline{w}(0,1) <b \right) >0
\end{equation}
for~ $0<a<b<\infty$. Then, for ${\mathcal Q}$-almost every $w\in {\mathcal W}$, there exists an ${\mathbb N}$-valued increasing sequence $\{ n_\ell ^w\}_{\ell \ge 1}$ such that
$$
\frac{1}{k^{\gamma}} \int _{{\mathbb R}^d} \left|u_{n_{\ell (k)}^w}(x)\right|^2 e^{-w(x)}{\rm d}x=o\left(\exp \left((b(1-\gamma) -a+\varepsilon)r^{\alpha (n_{\ell (k)}^w-1)}\right)\right) \quad (k \to \infty)
$$
for any $\varepsilon >0$ and $\gamma \in [0,1]$.
\end{prop}

\begin{rem}
The assumption (\ref{mainassump}) implies that the probability that $w$ looks like in the following figure is positive.

\vspace{3mm}
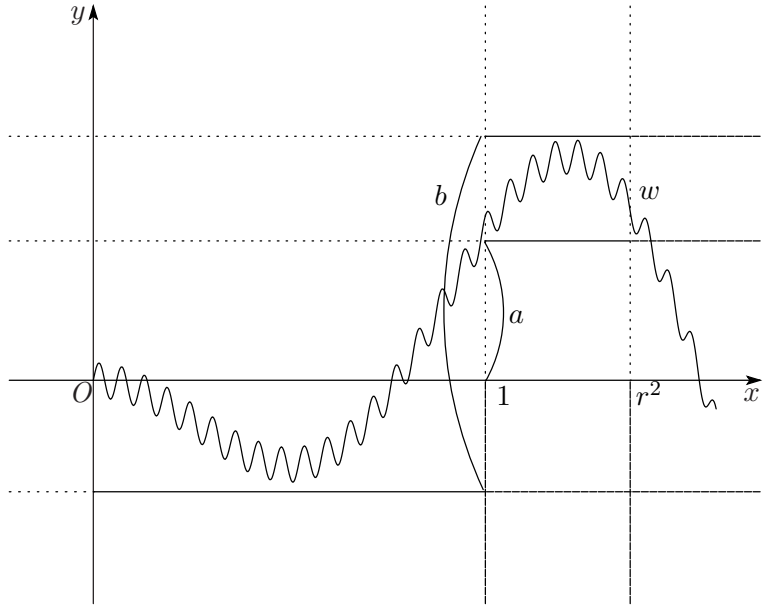
\begin{center}
{\unitlength 0.1in%
\begin{picture}(39.3000,31.3000)(3.4000,-34.5000)%
\put(7.7000,-22.9000){\makebox(0,0)[rt]{O}}%
\put(7.4000,-3.2000){\makebox(0,0)[rt]{$y$}}%
\put(42.7000,-23.2000){\makebox(0,0)[rt]{$x$}}%
%
\special{pn 8}%
\special{pa 780 3450}%
\special{pa 780 320}%
\special{fp}%
\special{sh 1}%
\special{pa 780 320}%
\special{pa 760 387}%
\special{pa 780 373}%
\special{pa 800 387}%
\special{pa 780 320}%
\special{fp}%
%
\special{pn 8}%
\special{pa 340 2280}%
\special{pa 4270 2280}%
\special{fp}%
\special{sh 1}%
\special{pa 4270 2280}%
\special{pa 4203 2260}%
\special{pa 4217 2280}%
\special{pa 4203 2300}%
\special{pa 4270 2280}%
\special{fp}%
\special{pn 8}%
\special{pn 8}%
\special{pa 340 2357}%
\special{pa 344 2364}%
\special{ip}%
\special{pa 355 2399}%
\special{pa 357 2407}%
\special{ip}%
\special{pa 366 2443}%
\special{pa 368 2451}%
\special{ip}%
\special{pa 379 2487}%
\special{pa 380 2490}%
\special{pa 382 2494}%
\special{ip}%
\special{pa 403 2479}%
\special{pa 405 2472}%
\special{pa 405 2471}%
\special{ip}%
\special{pa 414 2435}%
\special{pa 415 2428}%
\special{pa 415 2427}%
\special{ip}%
\special{pa 422 2391}%
\special{pa 424 2383}%
\special{ip}%
\special{pa 431 2346}%
\special{pa 432 2338}%
\special{ip}%
\special{pa 442 2302}%
\special{pa 444 2295}%
\special{ip}%
\special{pa 466 2301}%
\special{pa 469 2309}%
\special{ip}%
\special{pa 478 2345}%
\special{pa 480 2352}%
\special{ip}%
\special{pa 488 2388}%
\special{pa 490 2395}%
\special{pa 490 2396}%
\special{ip}%
\special{pa 503 2431}%
\special{pa 505 2434}%
\special{pa 509 2436}%
\special{ip}%
\special{pa 526 2406}%
\special{pa 528 2398}%
\special{ip}%
\special{pa 535 2362}%
\special{pa 537 2354}%
\special{ip}%
\special{pa 544 2317}%
\special{pa 545 2313}%
\special{pa 546 2309}%
\special{ip}%
\special{pa 553 2273}%
\special{pa 555 2266}%
\special{pa 555 2265}%
\special{ip}%
\special{pa 569 2231}%
\special{pa 570 2230}%
\special{pa 575 2229}%
\special{pa 576 2230}%
\special{ip}%
\special{pa 590 2264}%
\special{pa 592 2272}%
\special{ip}%
\special{pa 601 2308}%
\special{pa 602 2316}%
\special{ip}%
\special{pa 611 2352}%
\special{pa 613 2360}%
\special{ip}%
\special{pa 630 2392}%
\special{pa 630 2392}%
\special{pa 635 2389}%
\special{pa 636 2387}%
\special{ip}%
\special{pa 648 2352}%
\special{pa 650 2344}%
\special{ip}%
\special{pa 658 2308}%
\special{pa 659 2300}%
\special{ip}%
\special{pa 667 2264}%
\special{pa 668 2256}%
\special{ip}%
\special{pa 678 2220}%
\special{pa 680 2212}%
\special{ip}%
\special{pa 701 2209}%
\special{pa 704 2216}%
\special{ip}%
\special{pa 713 2252}%
\special{pa 715 2259}%
\special{pa 715 2260}%
\special{ip}%
\special{pa 723 2296}%
\special{pa 725 2304}%
\special{ip}%
\special{pa 733 2340}%
\special{pa 735 2346}%
\special{pa 736 2348}%
\special{ip}%
\special{pa 756 2367}%
\special{pa 760 2360}%
\special{pa 760 2360}%
\special{ip}%
\special{pa 770 2324}%
\special{pa 772 2316}%
\special{ip}%
\special{pa 780 2280}%
\special{pa 785 2256}%
\special{pa 790 2234}%
\special{pa 795 2215}%
\special{pa 800 2201}%
\special{pa 805 2192}%
\special{pa 810 2190}%
\special{pa 815 2193}%
\special{pa 820 2203}%
\special{pa 825 2218}%
\special{pa 830 2238}%
\special{pa 835 2261}%
\special{pa 845 2309}%
\special{pa 850 2332}%
\special{pa 855 2351}%
\special{pa 860 2366}%
\special{pa 865 2375}%
\special{pa 870 2379}%
\special{pa 875 2376}%
\special{pa 880 2367}%
\special{pa 885 2353}%
\special{pa 890 2335}%
\special{pa 900 2291}%
\special{pa 905 2268}%
\special{pa 910 2247}%
\special{pa 915 2230}%
\special{pa 920 2218}%
\special{pa 925 2211}%
\special{pa 930 2210}%
\special{pa 935 2215}%
\special{pa 940 2227}%
\special{pa 945 2244}%
\special{pa 950 2265}%
\special{pa 955 2289}%
\special{pa 960 2315}%
\special{pa 965 2340}%
\special{pa 970 2363}%
\special{pa 975 2383}%
\special{pa 980 2398}%
\special{pa 985 2408}%
\special{pa 990 2412}%
\special{pa 995 2409}%
\special{pa 1000 2401}%
\special{pa 1005 2387}%
\special{pa 1010 2369}%
\special{pa 1020 2327}%
\special{pa 1025 2305}%
\special{pa 1030 2286}%
\special{pa 1035 2270}%
\special{pa 1040 2259}%
\special{pa 1045 2254}%
\special{pa 1050 2255}%
\special{pa 1055 2262}%
\special{pa 1060 2275}%
\special{pa 1065 2294}%
\special{pa 1070 2316}%
\special{pa 1075 2341}%
\special{pa 1080 2368}%
\special{pa 1085 2394}%
\special{pa 1090 2418}%
\special{pa 1095 2438}%
\special{pa 1100 2453}%
\special{pa 1105 2463}%
\special{pa 1110 2467}%
\special{pa 1115 2464}%
\special{pa 1120 2456}%
\special{pa 1125 2443}%
\special{pa 1130 2425}%
\special{pa 1135 2405}%
\special{pa 1145 2363}%
\special{pa 1150 2345}%
\special{pa 1155 2331}%
\special{pa 1160 2321}%
\special{pa 1165 2317}%
\special{pa 1170 2320}%
\special{pa 1175 2328}%
\special{pa 1180 2343}%
\special{pa 1185 2362}%
\special{pa 1190 2386}%
\special{pa 1195 2412}%
\special{pa 1200 2439}%
\special{pa 1205 2465}%
\special{pa 1210 2489}%
\special{pa 1215 2510}%
\special{pa 1220 2525}%
\special{pa 1225 2535}%
\special{pa 1230 2538}%
\special{pa 1235 2535}%
\special{pa 1240 2527}%
\special{pa 1245 2513}%
\special{pa 1250 2496}%
\special{pa 1255 2476}%
\special{pa 1260 2455}%
\special{pa 1265 2435}%
\special{pa 1270 2418}%
\special{pa 1275 2404}%
\special{pa 1280 2396}%
\special{pa 1285 2393}%
\special{pa 1290 2397}%
\special{pa 1295 2406}%
\special{pa 1300 2422}%
\special{pa 1305 2442}%
\special{pa 1310 2466}%
\special{pa 1325 2547}%
\special{pa 1330 2571}%
\special{pa 1335 2590}%
\special{pa 1340 2605}%
\special{pa 1345 2614}%
\special{pa 1350 2617}%
\special{pa 1355 2614}%
\special{pa 1360 2605}%
\special{pa 1365 2590}%
\special{pa 1370 2573}%
\special{pa 1375 2553}%
\special{pa 1380 2532}%
\special{pa 1385 2512}%
\special{pa 1390 2495}%
\special{pa 1395 2482}%
\special{pa 1400 2474}%
\special{pa 1405 2472}%
\special{pa 1410 2477}%
\special{pa 1415 2487}%
\special{pa 1420 2503}%
\special{pa 1425 2524}%
\special{pa 1430 2548}%
\special{pa 1440 2602}%
\special{pa 1445 2628}%
\special{pa 1450 2651}%
\special{pa 1455 2670}%
\special{pa 1460 2684}%
\special{pa 1465 2692}%
\special{pa 1470 2694}%
\special{pa 1475 2689}%
\special{pa 1480 2679}%
\special{pa 1485 2664}%
\special{pa 1490 2646}%
\special{pa 1500 2604}%
\special{pa 1505 2584}%
\special{pa 1510 2567}%
\special{pa 1515 2554}%
\special{pa 1520 2547}%
\special{pa 1525 2545}%
\special{pa 1530 2549}%
\special{pa 1535 2560}%
\special{pa 1540 2576}%
\special{pa 1545 2597}%
\special{pa 1550 2621}%
\special{pa 1555 2647}%
\special{pa 1560 2674}%
\special{pa 1565 2699}%
\special{pa 1570 2721}%
\special{pa 1575 2739}%
\special{pa 1580 2751}%
\special{pa 1585 2758}%
\special{pa 1590 2758}%
\special{pa 1595 2752}%
\special{pa 1600 2740}%
\special{pa 1605 2724}%
\special{pa 1610 2705}%
\special{pa 1620 2661}%
\special{pa 1625 2641}%
\special{pa 1630 2623}%
\special{pa 1635 2610}%
\special{pa 1640 2602}%
\special{pa 1645 2600}%
\special{pa 1650 2605}%
\special{pa 1655 2615}%
\special{pa 1660 2631}%
\special{pa 1665 2651}%
\special{pa 1670 2675}%
\special{pa 1680 2725}%
\special{pa 1685 2749}%
\special{pa 1690 2770}%
\special{pa 1695 2786}%
\special{pa 1700 2797}%
\special{pa 1705 2801}%
\special{pa 1710 2799}%
\special{pa 1715 2792}%
\special{pa 1720 2778}%
\special{pa 1725 2760}%
\special{pa 1730 2739}%
\special{pa 1735 2716}%
\special{pa 1745 2672}%
\special{pa 1750 2654}%
\special{pa 1755 2640}%
\special{pa 1760 2632}%
\special{pa 1765 2629}%
\special{pa 1770 2633}%
\special{pa 1775 2643}%
\special{pa 1780 2658}%
\special{pa 1785 2677}%
\special{pa 1790 2700}%
\special{pa 1800 2748}%
\special{pa 1805 2770}%
\special{pa 1810 2788}%
\special{pa 1815 2803}%
\special{pa 1820 2811}%
\special{pa 1825 2814}%
\special{pa 1830 2810}%
\special{pa 1835 2800}%
\special{pa 1840 2784}%
\special{pa 1845 2764}%
\special{pa 1850 2742}%
\special{pa 1855 2717}%
\special{pa 1860 2693}%
\special{pa 1865 2670}%
\special{pa 1870 2651}%
\special{pa 1875 2636}%
\special{pa 1880 2627}%
\special{pa 1885 2624}%
\special{pa 1890 2627}%
\special{pa 1895 2636}%
\special{pa 1900 2650}%
\special{pa 1905 2668}%
\special{pa 1910 2689}%
\special{pa 1915 2712}%
\special{pa 1920 2734}%
\special{pa 1925 2754}%
\special{pa 1930 2771}%
\special{pa 1935 2782}%
\special{pa 1940 2789}%
\special{pa 1945 2789}%
\special{pa 1950 2782}%
\special{pa 1955 2770}%
\special{pa 1960 2752}%
\special{pa 1965 2730}%
\special{pa 1970 2705}%
\special{pa 1980 2653}%
\special{pa 1985 2629}%
\special{pa 1990 2609}%
\special{pa 1995 2593}%
\special{pa 2000 2583}%
\special{pa 2005 2579}%
\special{pa 2010 2581}%
\special{pa 2015 2589}%
\special{pa 2020 2602}%
\special{pa 2025 2619}%
\special{pa 2030 2638}%
\special{pa 2035 2659}%
\special{pa 2040 2679}%
\special{pa 2045 2697}%
\special{pa 2050 2711}%
\special{pa 2055 2721}%
\special{pa 2060 2725}%
\special{pa 2065 2722}%
\special{pa 2070 2713}%
\special{pa 2075 2698}%
\special{pa 2080 2678}%
\special{pa 2085 2654}%
\special{pa 2090 2627}%
\special{pa 2095 2599}%
\special{pa 2100 2572}%
\special{pa 2105 2547}%
\special{pa 2110 2525}%
\special{pa 2115 2508}%
\special{pa 2120 2497}%
\special{pa 2125 2492}%
\special{pa 2130 2493}%
\special{pa 2135 2499}%
\special{pa 2140 2511}%
\special{pa 2145 2527}%
\special{pa 2150 2545}%
\special{pa 2155 2564}%
\special{pa 2160 2582}%
\special{pa 2165 2598}%
\special{pa 2170 2610}%
\special{pa 2175 2617}%
\special{pa 2180 2619}%
\special{pa 2185 2614}%
\special{pa 2190 2602}%
\special{pa 2195 2585}%
\special{pa 2200 2563}%
\special{pa 2205 2537}%
\special{pa 2210 2508}%
\special{pa 2215 2478}%
\special{pa 2220 2449}%
\special{pa 2225 2423}%
\special{pa 2230 2400}%
\special{pa 2235 2382}%
\special{pa 2240 2370}%
\special{pa 2245 2364}%
\special{pa 2250 2364}%
\special{pa 2255 2370}%
\special{pa 2260 2381}%
\special{pa 2265 2395}%
\special{pa 2280 2446}%
\special{pa 2285 2460}%
\special{pa 2290 2470}%
\special{pa 2295 2475}%
\special{pa 2300 2474}%
\special{pa 2305 2466}%
\special{pa 2310 2453}%
\special{pa 2315 2433}%
\special{pa 2320 2409}%
\special{pa 2325 2381}%
\special{pa 2330 2351}%
\special{pa 2340 2289}%
\special{pa 2345 2262}%
\special{pa 2350 2238}%
\special{pa 2355 2220}%
\special{pa 2360 2207}%
\special{pa 2365 2200}%
\special{pa 2370 2200}%
\special{pa 2375 2205}%
\special{pa 2380 2215}%
\special{pa 2385 2229}%
\special{pa 2390 2244}%
\special{pa 2400 2276}%
\special{pa 2405 2288}%
\special{pa 2410 2296}%
\special{pa 2415 2299}%
\special{pa 2420 2296}%
\special{pa 2425 2287}%
\special{pa 2430 2271}%
\special{pa 2435 2250}%
\special{pa 2440 2224}%
\special{pa 2445 2195}%
\special{pa 2455 2131}%
\special{pa 2460 2100}%
\special{pa 2465 2072}%
\special{pa 2470 2048}%
\special{pa 2475 2029}%
\special{pa 2480 2016}%
\special{pa 2485 2009}%
\special{pa 2490 2008}%
\special{pa 2495 2013}%
\special{pa 2500 2023}%
\special{pa 2505 2036}%
\special{pa 2510 2051}%
\special{pa 2515 2067}%
\special{pa 2520 2080}%
\special{pa 2525 2092}%
\special{pa 2530 2098}%
\special{pa 2535 2100}%
\special{pa 2540 2096}%
\special{pa 2545 2085}%
\special{pa 2550 2068}%
\special{pa 2555 2045}%
\special{pa 2560 2018}%
\special{pa 2565 1988}%
\special{pa 2570 1956}%
\special{pa 2575 1923}%
\special{pa 2580 1892}%
\special{pa 2585 1863}%
\special{pa 2590 1839}%
\special{pa 2595 1821}%
\special{pa 2600 1808}%
\special{pa 2605 1802}%
\special{pa 2610 1802}%
\special{pa 2615 1807}%
\special{pa 2620 1817}%
\special{pa 2625 1830}%
\special{pa 2635 1860}%
\special{pa 2640 1873}%
\special{pa 2645 1883}%
\special{pa 2650 1890}%
\special{pa 2655 1890}%
\special{pa 2660 1885}%
\special{pa 2665 1873}%
\special{pa 2670 1855}%
\special{pa 2675 1832}%
\special{pa 2680 1805}%
\special{pa 2685 1774}%
\special{pa 2690 1741}%
\special{pa 2695 1709}%
\special{pa 2700 1678}%
\special{pa 2705 1650}%
\special{pa 2710 1627}%
\special{pa 2715 1609}%
\special{pa 2720 1598}%
\special{pa 2725 1592}%
\special{pa 2730 1593}%
\special{pa 2735 1599}%
\special{pa 2740 1610}%
\special{pa 2745 1623}%
\special{pa 2750 1639}%
\special{pa 2755 1654}%
\special{pa 2760 1667}%
\special{pa 2765 1678}%
\special{pa 2770 1684}%
\special{pa 2775 1684}%
\special{pa 2780 1678}%
\special{pa 2785 1666}%
\special{pa 2790 1649}%
\special{pa 2795 1625}%
\special{pa 2800 1598}%
\special{pa 2805 1567}%
\special{pa 2810 1535}%
\special{pa 2815 1504}%
\special{pa 2820 1474}%
\special{pa 2825 1447}%
\special{pa 2830 1426}%
\special{pa 2835 1409}%
\special{pa 2840 1399}%
\special{pa 2845 1395}%
\special{pa 2850 1398}%
\special{pa 2855 1405}%
\special{pa 2860 1417}%
\special{pa 2865 1432}%
\special{pa 2870 1449}%
\special{pa 2875 1465}%
\special{pa 2880 1479}%
\special{pa 2885 1490}%
\special{pa 2890 1496}%
\special{pa 2895 1497}%
\special{pa 2900 1492}%
\special{pa 2905 1480}%
\special{pa 2910 1463}%
\special{pa 2915 1440}%
\special{pa 2920 1414}%
\special{pa 2925 1384}%
\special{pa 2930 1353}%
\special{pa 2935 1323}%
\special{pa 2940 1295}%
\special{pa 2945 1270}%
\special{pa 2950 1251}%
\special{pa 2955 1237}%
\special{pa 2960 1229}%
\special{pa 2965 1227}%
\special{pa 2970 1232}%
\special{pa 2975 1241}%
\special{pa 2980 1255}%
\special{pa 2985 1272}%
\special{pa 2995 1308}%
\special{pa 3000 1323}%
\special{pa 3005 1336}%
\special{pa 3010 1343}%
\special{pa 3015 1345}%
\special{pa 3020 1340}%
\special{pa 3025 1330}%
\special{pa 3030 1314}%
\special{pa 3035 1292}%
\special{pa 3040 1267}%
\special{pa 3050 1211}%
\special{pa 3055 1182}%
\special{pa 3060 1157}%
\special{pa 3065 1135}%
\special{pa 3070 1117}%
\special{pa 3075 1106}%
\special{pa 3080 1101}%
\special{pa 3085 1102}%
\special{pa 3090 1110}%
\special{pa 3095 1122}%
\special{pa 3100 1139}%
\special{pa 3105 1158}%
\special{pa 3115 1198}%
\special{pa 3120 1215}%
\special{pa 3125 1229}%
\special{pa 3130 1238}%
\special{pa 3135 1242}%
\special{pa 3140 1239}%
\special{pa 3145 1230}%
\special{pa 3150 1215}%
\special{pa 3155 1196}%
\special{pa 3160 1173}%
\special{pa 3175 1095}%
\special{pa 3180 1072}%
\special{pa 3185 1053}%
\special{pa 3190 1039}%
\special{pa 3195 1031}%
\special{pa 3200 1030}%
\special{pa 3205 1034}%
\special{pa 3210 1045}%
\special{pa 3215 1061}%
\special{pa 3220 1080}%
\special{pa 3225 1102}%
\special{pa 3230 1125}%
\special{pa 3235 1147}%
\special{pa 3240 1167}%
\special{pa 3245 1183}%
\special{pa 3250 1194}%
\special{pa 3255 1200}%
\special{pa 3260 1199}%
\special{pa 3265 1192}%
\special{pa 3270 1180}%
\special{pa 3275 1162}%
\special{pa 3280 1142}%
\special{pa 3285 1119}%
\special{pa 3290 1095}%
\special{pa 3295 1073}%
\special{pa 3300 1053}%
\special{pa 3305 1038}%
\special{pa 3310 1027}%
\special{pa 3315 1023}%
\special{pa 3320 1025}%
\special{pa 3325 1033}%
\special{pa 3330 1047}%
\special{pa 3335 1067}%
\special{pa 3340 1090}%
\special{pa 3345 1115}%
\special{pa 3350 1141}%
\special{pa 3355 1166}%
\special{pa 3360 1188}%
\special{pa 3365 1207}%
\special{pa 3370 1220}%
\special{pa 3375 1228}%
\special{pa 3380 1229}%
\special{pa 3385 1225}%
\special{pa 3390 1215}%
\special{pa 3395 1200}%
\special{pa 3400 1182}%
\special{pa 3405 1162}%
\special{pa 3410 1141}%
\special{pa 3415 1122}%
\special{pa 3420 1106}%
\special{pa 3425 1094}%
\special{pa 3430 1088}%
\special{pa 3435 1087}%
\special{pa 3440 1093}%
\special{pa 3445 1105}%
\special{pa 3450 1123}%
\special{pa 3455 1146}%
\special{pa 3460 1173}%
\special{pa 3465 1201}%
\special{pa 3470 1230}%
\special{pa 3475 1258}%
\special{pa 3480 1283}%
\special{pa 3485 1304}%
\special{pa 3490 1320}%
\special{pa 3495 1330}%
\special{pa 3500 1334}%
\special{pa 3505 1332}%
\special{pa 3510 1324}%
\special{pa 3515 1312}%
\special{pa 3520 1297}%
\special{pa 3530 1263}%
\special{pa 3535 1247}%
\special{pa 3540 1234}%
\special{pa 3545 1226}%
\special{pa 3550 1224}%
\special{pa 3555 1227}%
\special{pa 3560 1237}%
\special{pa 3565 1253}%
\special{pa 3570 1275}%
\special{pa 3575 1301}%
\special{pa 3580 1331}%
\special{pa 3590 1395}%
\special{pa 3595 1425}%
\special{pa 3600 1453}%
\special{pa 3605 1476}%
\special{pa 3610 1495}%
\special{pa 3615 1507}%
\special{pa 3620 1513}%
\special{pa 3625 1513}%
\special{pa 3630 1508}%
\special{pa 3635 1498}%
\special{pa 3640 1486}%
\special{pa 3645 1472}%
\special{pa 3650 1457}%
\special{pa 3655 1445}%
\special{pa 3660 1436}%
\special{pa 3665 1431}%
\special{pa 3670 1432}%
\special{pa 3675 1439}%
\special{pa 3680 1453}%
\special{pa 3685 1472}%
\special{pa 3690 1498}%
\special{pa 3695 1527}%
\special{pa 3700 1560}%
\special{pa 3710 1630}%
\special{pa 3715 1663}%
\special{pa 3720 1693}%
\special{pa 3725 1718}%
\special{pa 3730 1738}%
\special{pa 3735 1753}%
\special{pa 3740 1761}%
\special{pa 3745 1763}%
\special{pa 3750 1760}%
\special{pa 3755 1752}%
\special{pa 3760 1742}%
\special{pa 3765 1730}%
\special{pa 3770 1719}%
\special{pa 3775 1709}%
\special{pa 3780 1703}%
\special{pa 3785 1702}%
\special{pa 3790 1706}%
\special{pa 3795 1716}%
\special{pa 3800 1733}%
\special{pa 3805 1756}%
\special{pa 3810 1784}%
\special{pa 3815 1816}%
\special{pa 3820 1852}%
\special{pa 3830 1926}%
\special{pa 3835 1961}%
\special{pa 3840 1993}%
\special{pa 3845 2020}%
\special{pa 3850 2042}%
\special{pa 3855 2057}%
\special{pa 3860 2067}%
\special{pa 3865 2070}%
\special{pa 3870 2069}%
\special{pa 3875 2063}%
\special{pa 3885 2045}%
\special{pa 3890 2035}%
\special{pa 3895 2028}%
\special{pa 3900 2024}%
\special{pa 3905 2025}%
\special{pa 3910 2032}%
\special{pa 3915 2045}%
\special{pa 3920 2064}%
\special{pa 3925 2090}%
\special{pa 3930 2120}%
\special{pa 3935 2155}%
\special{pa 3945 2231}%
\special{pa 3950 2270}%
\special{pa 3955 2306}%
\special{pa 3960 2339}%
\special{pa 3965 2367}%
\special{pa 3970 2389}%
\special{pa 3975 2406}%
\special{pa 3980 2416}%
\special{pa 3985 2420}%
\special{pa 3990 2419}%
\special{pa 3995 2414}%
\special{pa 4000 2407}%
\special{pa 4010 2391}%
\special{pa 4015 2385}%
\special{pa 4020 2383}%
\special{pa 4025 2386}%
\special{pa 4030 2394}%
\special{pa 4035 2409}%
\special{pa 4040 2430}%
\special{fp}%
\special{pn 8}%
\special{pa 4041 2438}%
\special{pa 4045 2457}%
\special{pa 4048 2475}%
\special{ip}%
\special{pa 4049 2483}%
\special{pa 4050 2489}%
\special{pa 4054 2520}%
\special{ip}%
\special{pa 4055 2529}%
\special{pa 4060 2564}%
\special{pa 4060 2566}%
\special{ip}%
\special{pa 4061 2574}%
\special{pa 4065 2604}%
\special{pa 4066 2611}%
\special{ip}%
\special{pa 4067 2619}%
\special{pa 4070 2643}%
\special{pa 4072 2657}%
\special{ip}%
\special{pa 4073 2665}%
\special{pa 4075 2680}%
\special{pa 4078 2702}%
\special{ip}%
\special{pa 4080 2710}%
\special{pa 4080 2713}%
\special{pa 4085 2741}%
\special{pa 4086 2747}%
\special{ip}%
\special{pa 4088 2755}%
\special{pa 4090 2763}%
\special{pa 4095 2779}%
\special{pa 4100 2789}%
\special{pa 4101 2790}%
\special{ip}%
\special{pa 4108 2792}%
\special{pa 4110 2792}%
\special{pa 4115 2788}%
\special{pa 4125 2772}%
\special{pa 4130 2765}%
\special{pa 4131 2764}%
\special{ip}%
\special{pa 4138 2759}%
\special{pa 4140 2759}%
\special{pa 4145 2762}%
\special{pa 4150 2772}%
\special{pa 4155 2788}%
\special{pa 4155 2790}%
\special{ip}%
\special{pa 4157 2798}%
\special{pa 4160 2810}%
\special{pa 4165 2835}%
\special{ip}%
\special{pa 4166 2843}%
\special{pa 4170 2870}%
\special{pa 4171 2880}%
\special{ip}%
\special{pa 4172 2888}%
\special{pa 4175 2907}%
\special{pa 4177 2926}%
\special{ip}%
\special{pa 4178 2934}%
\special{pa 4183 2971}%
\special{ip}%
\special{pa 4184 2979}%
\special{pa 4189 3016}%
\special{ip}%
\special{pa 4190 3024}%
\special{pa 4195 3060}%
\special{pa 4195 3062}%
\special{ip}%
\special{pa 4197 3070}%
\special{pa 4200 3092}%
\special{pa 4203 3107}%
\special{ip}%
\special{pa 4204 3115}%
\special{pa 4205 3120}%
\special{pa 4210 3141}%
\special{pa 4213 3151}%
\special{ip}%
\special{pa 4217 3159}%
\special{pa 4220 3165}%
\special{pa 4225 3168}%
\special{pa 4230 3166}%
\special{pa 4235 3160}%
\special{pa 4241 3150}%
\special{ip}%
\special{pa 4245 3143}%
\special{pa 4250 3136}%
\special{pa 4255 3131}%
\special{pa 4260 3129}%
\special{pa 4265 3133}%
\special{pa 4270 3142}%
\special{ip}%
\special{pn 8}%
\special{pn 8}%
\special{pa 340 2864}%
\special{pa 349 2864}%
\special{fp}%
\special{pa 389 2864}%
\special{pa 398 2864}%
\special{fp}%
\special{pa 438 2864}%
\special{pa 446 2864}%
\special{fp}%
\special{pa 487 2864}%
\special{pa 495 2864}%
\special{fp}%
\special{pa 536 2864}%
\special{pa 544 2864}%
\special{fp}%
\special{pa 584 2864}%
\special{pa 593 2864}%
\special{fp}%
\special{pa 633 2864}%
\special{pa 642 2864}%
\special{fp}%
\special{pa 682 2864}%
\special{pa 691 2864}%
\special{fp}%
\special{pa 731 2864}%
\special{pa 740 2864}%
\special{fp}%
\special{pa 780 2864}%
\special{pa 780 2864}%
\special{fp}%
\special{pa 780 2864}%
\special{pa 2830 2864}%
\special{fp}%
\special{pn 8}%
\special{pa 2838 2864}%
\special{pa 2875 2864}%
\special{fp}%
\special{pa 2883 2864}%
\special{pa 2920 2864}%
\special{fp}%
\special{pa 2928 2864}%
\special{pa 2965 2864}%
\special{fp}%
\special{pa 2973 2864}%
\special{pa 3010 2864}%
\special{fp}%
\special{pa 3018 2864}%
\special{pa 3055 2864}%
\special{fp}%
\special{pa 3063 2864}%
\special{pa 3100 2864}%
\special{fp}%
\special{pa 3108 2864}%
\special{pa 3145 2864}%
\special{fp}%
\special{pa 3153 2864}%
\special{pa 3190 2864}%
\special{fp}%
\special{pa 3198 2864}%
\special{pa 3235 2864}%
\special{fp}%
\special{pa 3243 2864}%
\special{pa 3280 2864}%
\special{fp}%
\special{pa 3288 2864}%
\special{pa 3325 2864}%
\special{fp}%
\special{pa 3333 2864}%
\special{pa 3370 2864}%
\special{fp}%
\special{pa 3378 2864}%
\special{pa 3415 2864}%
\special{fp}%
\special{pa 3423 2864}%
\special{pa 3460 2864}%
\special{fp}%
\special{pa 3468 2864}%
\special{pa 3505 2864}%
\special{fp}%
\special{pa 3513 2864}%
\special{pa 3550 2864}%
\special{fp}%
\special{pa 3558 2864}%
\special{pa 3595 2864}%
\special{fp}%
\special{pa 3603 2864}%
\special{pa 3640 2864}%
\special{fp}%
\special{pa 3648 2864}%
\special{pa 3685 2864}%
\special{fp}%
\special{pa 3693 2864}%
\special{pa 3730 2864}%
\special{fp}%
\special{pa 3738 2864}%
\special{pa 3775 2864}%
\special{fp}%
\special{pa 3783 2864}%
\special{pa 3820 2864}%
\special{fp}%
\special{pa 3828 2864}%
\special{pa 3865 2864}%
\special{fp}%
\special{pa 3873 2864}%
\special{pa 3910 2864}%
\special{fp}%
\special{pa 3918 2864}%
\special{pa 3955 2864}%
\special{fp}%
\special{pa 3963 2864}%
\special{pa 4000 2864}%
\special{fp}%
\special{pa 4008 2864}%
\special{pa 4045 2864}%
\special{fp}%
\special{pa 4053 2864}%
\special{pa 4090 2864}%
\special{fp}%
\special{pa 4098 2864}%
\special{pa 4135 2864}%
\special{fp}%
\special{pa 4143 2864}%
\special{pa 4180 2864}%
\special{fp}%
\special{pa 4188 2864}%
\special{pa 4225 2864}%
\special{fp}%
\special{pa 4233 2864}%
\special{pa 4270 2864}%
\special{fp}%
\special{pn 8}%
\special{pn 8}%
\special{pa 340 1550}%
\special{pa 348 1550}%
\special{fp}%
\special{pa 385 1550}%
\special{pa 393 1550}%
\special{fp}%
\special{pa 431 1550}%
\special{pa 439 1550}%
\special{fp}%
\special{pa 476 1550}%
\special{pa 484 1550}%
\special{fp}%
\special{pa 521 1550}%
\special{pa 529 1550}%
\special{fp}%
\special{pa 566 1550}%
\special{pa 574 1550}%
\special{fp}%
\special{pa 612 1550}%
\special{pa 620 1550}%
\special{fp}%
\special{pa 657 1550}%
\special{pa 665 1550}%
\special{fp}%
\special{pa 702 1550}%
\special{pa 710 1550}%
\special{fp}%
\special{pa 747 1550}%
\special{pa 756 1550}%
\special{fp}%
\special{pa 793 1550}%
\special{pa 801 1550}%
\special{fp}%
\special{pa 838 1550}%
\special{pa 846 1550}%
\special{fp}%
\special{pa 883 1550}%
\special{pa 891 1550}%
\special{fp}%
\special{pa 929 1550}%
\special{pa 937 1550}%
\special{fp}%
\special{pa 974 1550}%
\special{pa 982 1550}%
\special{fp}%
\special{pa 1019 1550}%
\special{pa 1027 1550}%
\special{fp}%
\special{pa 1064 1550}%
\special{pa 1072 1550}%
\special{fp}%
\special{pa 1110 1550}%
\special{pa 1118 1550}%
\special{fp}%
\special{pa 1155 1550}%
\special{pa 1163 1550}%
\special{fp}%
\special{pa 1200 1550}%
\special{pa 1208 1550}%
\special{fp}%
\special{pa 1245 1550}%
\special{pa 1254 1550}%
\special{fp}%
\special{pa 1291 1550}%
\special{pa 1299 1550}%
\special{fp}%
\special{pa 1336 1550}%
\special{pa 1344 1550}%
\special{fp}%
\special{pa 1381 1550}%
\special{pa 1389 1550}%
\special{fp}%
\special{pa 1427 1550}%
\special{pa 1435 1550}%
\special{fp}%
\special{pa 1472 1550}%
\special{pa 1480 1550}%
\special{fp}%
\special{pa 1517 1550}%
\special{pa 1525 1550}%
\special{fp}%
\special{pa 1562 1550}%
\special{pa 1570 1550}%
\special{fp}%
\special{pa 1608 1550}%
\special{pa 1616 1550}%
\special{fp}%
\special{pa 1653 1550}%
\special{pa 1661 1550}%
\special{fp}%
\special{pa 1698 1550}%
\special{pa 1706 1550}%
\special{fp}%
\special{pa 1743 1550}%
\special{pa 1752 1550}%
\special{fp}%
\special{pa 1789 1550}%
\special{pa 1797 1550}%
\special{fp}%
\special{pa 1834 1550}%
\special{pa 1842 1550}%
\special{fp}%
\special{pa 1879 1550}%
\special{pa 1887 1550}%
\special{fp}%
\special{pa 1925 1550}%
\special{pa 1933 1550}%
\special{fp}%
\special{pa 1970 1550}%
\special{pa 1978 1550}%
\special{fp}%
\special{pa 2015 1550}%
\special{pa 2023 1550}%
\special{fp}%
\special{pa 2060 1550}%
\special{pa 2068 1550}%
\special{fp}%
\special{pa 2106 1550}%
\special{pa 2114 1550}%
\special{fp}%
\special{pa 2151 1550}%
\special{pa 2159 1550}%
\special{fp}%
\special{pa 2196 1550}%
\special{pa 2204 1550}%
\special{fp}%
\special{pa 2241 1550}%
\special{pa 2250 1550}%
\special{fp}%
\special{pa 2287 1550}%
\special{pa 2295 1550}%
\special{fp}%
\special{pa 2332 1550}%
\special{pa 2340 1550}%
\special{fp}%
\special{pa 2377 1550}%
\special{pa 2385 1550}%
\special{fp}%
\special{pa 2423 1550}%
\special{pa 2431 1550}%
\special{fp}%
\special{pa 2468 1550}%
\special{pa 2476 1550}%
\special{fp}%
\special{pa 2513 1550}%
\special{pa 2521 1550}%
\special{fp}%
\special{pa 2558 1550}%
\special{pa 2566 1550}%
\special{fp}%
\special{pa 2604 1550}%
\special{pa 2612 1550}%
\special{fp}%
\special{pa 2649 1550}%
\special{pa 2657 1550}%
\special{fp}%
\special{pa 2694 1550}%
\special{pa 2702 1550}%
\special{fp}%
\special{pa 2739 1550}%
\special{pa 2748 1550}%
\special{fp}%
\special{pa 2785 1550}%
\special{pa 2793 1550}%
\special{fp}%
\special{pa 2830 1550}%
\special{pa 2830 1550}%
\special{fp}%
\special{pa 2830 1550}%
\special{pa 3590 1550}%
\special{fp}%
\special{pn 8}%
\special{pa 3598 1550}%
\special{pa 3635 1550}%
\special{fp}%
\special{pa 3643 1550}%
\special{pa 3681 1550}%
\special{fp}%
\special{pa 3689 1550}%
\special{pa 3726 1550}%
\special{fp}%
\special{pa 3734 1550}%
\special{pa 3771 1550}%
\special{fp}%
\special{pa 3779 1550}%
\special{pa 3817 1550}%
\special{fp}%
\special{pa 3825 1550}%
\special{pa 3862 1550}%
\special{fp}%
\special{pa 3870 1550}%
\special{pa 3907 1550}%
\special{fp}%
\special{pa 3915 1550}%
\special{pa 3953 1550}%
\special{fp}%
\special{pa 3961 1550}%
\special{pa 3998 1550}%
\special{fp}%
\special{pa 4006 1550}%
\special{pa 4043 1550}%
\special{fp}%
\special{pa 4051 1550}%
\special{pa 4089 1550}%
\special{fp}%
\special{pa 4097 1550}%
\special{pa 4134 1550}%
\special{fp}%
\special{pa 4142 1550}%
\special{pa 4179 1550}%
\special{fp}%
\special{pa 4187 1550}%
\special{pa 4225 1550}%
\special{fp}%
\special{pa 4233 1550}%
\special{pa 4270 1550}%
\special{fp}%
\special{pn 8}%
\special{pn 8}%
\special{pa 340 1002}%
\special{pa 348 1002}%
\special{fp}%
\special{pa 385 1002}%
\special{pa 393 1002}%
\special{fp}%
\special{pa 431 1002}%
\special{pa 439 1002}%
\special{fp}%
\special{pa 476 1002}%
\special{pa 484 1002}%
\special{fp}%
\special{pa 521 1002}%
\special{pa 529 1002}%
\special{fp}%
\special{pa 566 1002}%
\special{pa 574 1002}%
\special{fp}%
\special{pa 612 1002}%
\special{pa 620 1002}%
\special{fp}%
\special{pa 657 1002}%
\special{pa 665 1002}%
\special{fp}%
\special{pa 702 1002}%
\special{pa 710 1002}%
\special{fp}%
\special{pa 747 1002}%
\special{pa 756 1002}%
\special{fp}%
\special{pa 793 1002}%
\special{pa 801 1002}%
\special{fp}%
\special{pa 838 1002}%
\special{pa 846 1002}%
\special{fp}%
\special{pa 883 1002}%
\special{pa 891 1002}%
\special{fp}%
\special{pa 929 1002}%
\special{pa 937 1002}%
\special{fp}%
\special{pa 974 1002}%
\special{pa 982 1002}%
\special{fp}%
\special{pa 1019 1002}%
\special{pa 1027 1002}%
\special{fp}%
\special{pa 1064 1002}%
\special{pa 1072 1002}%
\special{fp}%
\special{pa 1110 1002}%
\special{pa 1118 1002}%
\special{fp}%
\special{pa 1155 1002}%
\special{pa 1163 1002}%
\special{fp}%
\special{pa 1200 1002}%
\special{pa 1208 1002}%
\special{fp}%
\special{pa 1245 1002}%
\special{pa 1254 1002}%
\special{fp}%
\special{pa 1291 1002}%
\special{pa 1299 1002}%
\special{fp}%
\special{pa 1336 1002}%
\special{pa 1344 1002}%
\special{fp}%
\special{pa 1381 1002}%
\special{pa 1389 1002}%
\special{fp}%
\special{pa 1427 1002}%
\special{pa 1435 1002}%
\special{fp}%
\special{pa 1472 1002}%
\special{pa 1480 1002}%
\special{fp}%
\special{pa 1517 1002}%
\special{pa 1525 1002}%
\special{fp}%
\special{pa 1562 1002}%
\special{pa 1570 1002}%
\special{fp}%
\special{pa 1608 1002}%
\special{pa 1616 1002}%
\special{fp}%
\special{pa 1653 1002}%
\special{pa 1661 1002}%
\special{fp}%
\special{pa 1698 1002}%
\special{pa 1706 1002}%
\special{fp}%
\special{pa 1743 1002}%
\special{pa 1752 1002}%
\special{fp}%
\special{pa 1789 1002}%
\special{pa 1797 1002}%
\special{fp}%
\special{pa 1834 1002}%
\special{pa 1842 1002}%
\special{fp}%
\special{pa 1879 1002}%
\special{pa 1887 1002}%
\special{fp}%
\special{pa 1925 1002}%
\special{pa 1933 1002}%
\special{fp}%
\special{pa 1970 1002}%
\special{pa 1978 1002}%
\special{fp}%
\special{pa 2015 1002}%
\special{pa 2023 1002}%
\special{fp}%
\special{pa 2060 1002}%
\special{pa 2068 1002}%
\special{fp}%
\special{pa 2106 1002}%
\special{pa 2114 1002}%
\special{fp}%
\special{pa 2151 1002}%
\special{pa 2159 1002}%
\special{fp}%
\special{pa 2196 1002}%
\special{pa 2204 1002}%
\special{fp}%
\special{pa 2241 1002}%
\special{pa 2250 1002}%
\special{fp}%
\special{pa 2287 1002}%
\special{pa 2295 1002}%
\special{fp}%
\special{pa 2332 1002}%
\special{pa 2340 1002}%
\special{fp}%
\special{pa 2377 1002}%
\special{pa 2385 1002}%
\special{fp}%
\special{pa 2423 1002}%
\special{pa 2431 1002}%
\special{fp}%
\special{pa 2468 1002}%
\special{pa 2476 1002}%
\special{fp}%
\special{pa 2513 1002}%
\special{pa 2521 1002}%
\special{fp}%
\special{pa 2558 1002}%
\special{pa 2566 1002}%
\special{fp}%
\special{pa 2604 1002}%
\special{pa 2612 1002}%
\special{fp}%
\special{pa 2649 1002}%
\special{pa 2657 1002}%
\special{fp}%
\special{pa 2694 1002}%
\special{pa 2702 1002}%
\special{fp}%
\special{pa 2739 1002}%
\special{pa 2748 1002}%
\special{fp}%
\special{pa 2785 1002}%
\special{pa 2793 1002}%
\special{fp}%
\special{pa 2830 1002}%
\special{pa 2830 1002}%
\special{fp}%
\special{pa 2830 1002}%
\special{pa 3590 1002}%
\special{fp}%
\special{pn 8}%
\special{pa 3598 1002}%
\special{pa 3635 1002}%
\special{fp}%
\special{pa 3643 1002}%
\special{pa 3681 1002}%
\special{fp}%
\special{pa 3689 1002}%
\special{pa 3726 1002}%
\special{fp}%
\special{pa 3734 1002}%
\special{pa 3771 1002}%
\special{fp}%
\special{pa 3779 1002}%
\special{pa 3817 1002}%
\special{fp}%
\special{pa 3825 1002}%
\special{pa 3862 1002}%
\special{fp}%
\special{pa 3870 1002}%
\special{pa 3907 1002}%
\special{fp}%
\special{pa 3915 1002}%
\special{pa 3953 1002}%
\special{fp}%
\special{pa 3961 1002}%
\special{pa 3998 1002}%
\special{fp}%
\special{pa 4006 1002}%
\special{pa 4043 1002}%
\special{fp}%
\special{pa 4051 1002}%
\special{pa 4089 1002}%
\special{fp}%
\special{pa 4097 1002}%
\special{pa 4134 1002}%
\special{fp}%
\special{pa 4142 1002}%
\special{pa 4179 1002}%
\special{fp}%
\special{pa 4187 1002}%
\special{pa 4225 1002}%
\special{fp}%
\special{pa 4233 1002}%
\special{pa 4270 1002}%
\special{fp}%
\special{pn 8}%
\special{pn 8}%
\special{pa 2832 320}%
\special{pa 2832 328}%
\special{fp}%
\special{pa 2832 366}%
\special{pa 2832 374}%
\special{fp}%
\special{pa 2832 411}%
\special{pa 2832 419}%
\special{fp}%
\special{pa 2832 457}%
\special{pa 2832 465}%
\special{fp}%
\special{pa 2832 502}%
\special{pa 2832 510}%
\special{fp}%
\special{pa 2832 548}%
\special{pa 2832 556}%
\special{fp}%
\special{pa 2832 593}%
\special{pa 2832 602}%
\special{fp}%
\special{pa 2832 639}%
\special{pa 2832 647}%
\special{fp}%
\special{pa 2832 685}%
\special{pa 2832 693}%
\special{fp}%
\special{pa 2832 730}%
\special{pa 2832 738}%
\special{fp}%
\special{pa 2832 776}%
\special{pa 2832 784}%
\special{fp}%
\special{pa 2832 821}%
\special{pa 2832 829}%
\special{fp}%
\special{pa 2832 867}%
\special{pa 2832 875}%
\special{fp}%
\special{pa 2832 913}%
\special{pa 2832 921}%
\special{fp}%
\special{pa 2832 958}%
\special{pa 2832 966}%
\special{fp}%
\special{pa 2832 1004}%
\special{pa 2832 1012}%
\special{fp}%
\special{pa 2832 1049}%
\special{pa 2832 1057}%
\special{fp}%
\special{pa 2832 1095}%
\special{pa 2832 1103}%
\special{fp}%
\special{pa 2832 1140}%
\special{pa 2832 1149}%
\special{fp}%
\special{pa 2832 1186}%
\special{pa 2832 1194}%
\special{fp}%
\special{pa 2832 1232}%
\special{pa 2832 1240}%
\special{fp}%
\special{pa 2832 1277}%
\special{pa 2832 1285}%
\special{fp}%
\special{pa 2832 1323}%
\special{pa 2832 1331}%
\special{fp}%
\special{pa 2832 1368}%
\special{pa 2832 1376}%
\special{fp}%
\special{pa 2832 1414}%
\special{pa 2832 1422}%
\special{fp}%
\special{pa 2832 1460}%
\special{pa 2832 1468}%
\special{fp}%
\special{pa 2832 1505}%
\special{pa 2832 1513}%
\special{fp}%
\special{pa 2832 1551}%
\special{pa 2832 1559}%
\special{fp}%
\special{pa 2832 1596}%
\special{pa 2832 1604}%
\special{fp}%
\special{pa 2832 1642}%
\special{pa 2832 1650}%
\special{fp}%
\special{pa 2832 1687}%
\special{pa 2832 1696}%
\special{fp}%
\special{pa 2832 1733}%
\special{pa 2832 1741}%
\special{fp}%
\special{pa 2832 1779}%
\special{pa 2832 1787}%
\special{fp}%
\special{pa 2832 1824}%
\special{pa 2832 1832}%
\special{fp}%
\special{pa 2832 1870}%
\special{pa 2832 1878}%
\special{fp}%
\special{pa 2832 1915}%
\special{pa 2832 1923}%
\special{fp}%
\special{pa 2832 1961}%
\special{pa 2832 1969}%
\special{fp}%
\special{pa 2832 2007}%
\special{pa 2832 2015}%
\special{fp}%
\special{pa 2832 2052}%
\special{pa 2832 2060}%
\special{fp}%
\special{pa 2832 2098}%
\special{pa 2832 2106}%
\special{fp}%
\special{pa 2832 2143}%
\special{pa 2832 2151}%
\special{fp}%
\special{pa 2832 2189}%
\special{pa 2832 2197}%
\special{fp}%
\special{pa 2832 2234}%
\special{pa 2832 2243}%
\special{fp}%
\special{pa 2832 2280}%
\special{pa 2832 2280}%
\special{fp}%
\special{pn 8}%
\special{pa 2832 2288}%
\special{pa 2832 2327}%
\special{fp}%
\special{pa 2832 2335}%
\special{pa 2832 2374}%
\special{fp}%
\special{pa 2832 2382}%
\special{pa 2832 2420}%
\special{fp}%
\special{pa 2832 2429}%
\special{pa 2832 2467}%
\special{fp}%
\special{pa 2832 2476}%
\special{pa 2832 2514}%
\special{fp}%
\special{pa 2832 2522}%
\special{pa 2832 2561}%
\special{fp}%
\special{pa 2832 2569}%
\special{pa 2832 2608}%
\special{fp}%
\special{pa 2832 2616}%
\special{pa 2832 2654}%
\special{fp}%
\special{pa 2832 2663}%
\special{pa 2832 2701}%
\special{fp}%
\special{pa 2832 2710}%
\special{pa 2832 2748}%
\special{fp}%
\special{pa 2832 2756}%
\special{pa 2832 2795}%
\special{fp}%
\special{pa 2832 2803}%
\special{pa 2832 2842}%
\special{fp}%
\special{pa 2832 2850}%
\special{pa 2832 2888}%
\special{fp}%
\special{pa 2832 2897}%
\special{pa 2832 2935}%
\special{fp}%
\special{pa 2832 2944}%
\special{pa 2832 2982}%
\special{fp}%
\special{pa 2832 2990}%
\special{pa 2832 3029}%
\special{fp}%
\special{pa 2832 3037}%
\special{pa 2832 3076}%
\special{fp}%
\special{pa 2832 3084}%
\special{pa 2832 3122}%
\special{fp}%
\special{pa 2832 3131}%
\special{pa 2832 3169}%
\special{fp}%
\special{pa 2832 3178}%
\special{pa 2832 3216}%
\special{fp}%
\special{pa 2832 3224}%
\special{pa 2832 3263}%
\special{fp}%
\special{pa 2832 3271}%
\special{pa 2832 3310}%
\special{fp}%
\special{pa 2832 3318}%
\special{pa 2832 3356}%
\special{fp}%
\special{pa 2832 3365}%
\special{pa 2832 3403}%
\special{fp}%
\special{pa 2832 3412}%
\special{pa 2832 3450}%
\special{fp}%
\special{pn 8}%
\special{pn 8}%
\special{pa 3590 320}%
\special{pa 3590 328}%
\special{fp}%
\special{pa 3590 366}%
\special{pa 3590 374}%
\special{fp}%
\special{pa 3590 411}%
\special{pa 3590 419}%
\special{fp}%
\special{pa 3590 457}%
\special{pa 3590 465}%
\special{fp}%
\special{pa 3590 502}%
\special{pa 3590 510}%
\special{fp}%
\special{pa 3590 548}%
\special{pa 3590 556}%
\special{fp}%
\special{pa 3590 593}%
\special{pa 3590 602}%
\special{fp}%
\special{pa 3590 639}%
\special{pa 3590 647}%
\special{fp}%
\special{pa 3590 685}%
\special{pa 3590 693}%
\special{fp}%
\special{pa 3590 730}%
\special{pa 3590 738}%
\special{fp}%
\special{pa 3590 776}%
\special{pa 3590 784}%
\special{fp}%
\special{pa 3590 821}%
\special{pa 3590 829}%
\special{fp}%
\special{pa 3590 867}%
\special{pa 3590 875}%
\special{fp}%
\special{pa 3590 913}%
\special{pa 3590 921}%
\special{fp}%
\special{pa 3590 958}%
\special{pa 3590 966}%
\special{fp}%
\special{pa 3590 1004}%
\special{pa 3590 1012}%
\special{fp}%
\special{pa 3590 1049}%
\special{pa 3590 1057}%
\special{fp}%
\special{pa 3590 1095}%
\special{pa 3590 1103}%
\special{fp}%
\special{pa 3590 1140}%
\special{pa 3590 1149}%
\special{fp}%
\special{pa 3590 1186}%
\special{pa 3590 1194}%
\special{fp}%
\special{pa 3590 1232}%
\special{pa 3590 1240}%
\special{fp}%
\special{pa 3590 1277}%
\special{pa 3590 1285}%
\special{fp}%
\special{pa 3590 1323}%
\special{pa 3590 1331}%
\special{fp}%
\special{pa 3590 1368}%
\special{pa 3590 1376}%
\special{fp}%
\special{pa 3590 1414}%
\special{pa 3590 1422}%
\special{fp}%
\special{pa 3590 1460}%
\special{pa 3590 1468}%
\special{fp}%
\special{pa 3590 1505}%
\special{pa 3590 1513}%
\special{fp}%
\special{pa 3590 1551}%
\special{pa 3590 1559}%
\special{fp}%
\special{pa 3590 1596}%
\special{pa 3590 1604}%
\special{fp}%
\special{pa 3590 1642}%
\special{pa 3590 1650}%
\special{fp}%
\special{pa 3590 1687}%
\special{pa 3590 1696}%
\special{fp}%
\special{pa 3590 1733}%
\special{pa 3590 1741}%
\special{fp}%
\special{pa 3590 1779}%
\special{pa 3590 1787}%
\special{fp}%
\special{pa 3590 1824}%
\special{pa 3590 1832}%
\special{fp}%
\special{pa 3590 1870}%
\special{pa 3590 1878}%
\special{fp}%
\special{pa 3590 1915}%
\special{pa 3590 1923}%
\special{fp}%
\special{pa 3590 1961}%
\special{pa 3590 1969}%
\special{fp}%
\special{pa 3590 2007}%
\special{pa 3590 2015}%
\special{fp}%
\special{pa 3590 2052}%
\special{pa 3590 2060}%
\special{fp}%
\special{pa 3590 2098}%
\special{pa 3590 2106}%
\special{fp}%
\special{pa 3590 2143}%
\special{pa 3590 2151}%
\special{fp}%
\special{pa 3590 2189}%
\special{pa 3590 2197}%
\special{fp}%
\special{pa 3590 2234}%
\special{pa 3590 2243}%
\special{fp}%
\special{pa 3590 2280}%
\special{pa 3590 2280}%
\special{fp}%
\special{pn 8}%
\special{pa 3590 2288}%
\special{pa 3590 2327}%
\special{fp}%
\special{pa 3590 2335}%
\special{pa 3590 2374}%
\special{fp}%
\special{pa 3590 2382}%
\special{pa 3590 2420}%
\special{fp}%
\special{pa 3590 2429}%
\special{pa 3590 2467}%
\special{fp}%
\special{pa 3590 2476}%
\special{pa 3590 2514}%
\special{fp}%
\special{pa 3590 2522}%
\special{pa 3590 2561}%
\special{fp}%
\special{pa 3590 2569}%
\special{pa 3590 2608}%
\special{fp}%
\special{pa 3590 2616}%
\special{pa 3590 2654}%
\special{fp}%
\special{pa 3590 2663}%
\special{pa 3590 2701}%
\special{fp}%
\special{pa 3590 2710}%
\special{pa 3590 2748}%
\special{fp}%
\special{pa 3590 2756}%
\special{pa 3590 2795}%
\special{fp}%
\special{pa 3590 2803}%
\special{pa 3590 2842}%
\special{fp}%
\special{pa 3590 2850}%
\special{pa 3590 2888}%
\special{fp}%
\special{pa 3590 2897}%
\special{pa 3590 2935}%
\special{fp}%
\special{pa 3590 2944}%
\special{pa 3590 2982}%
\special{fp}%
\special{pa 3590 2990}%
\special{pa 3590 3029}%
\special{fp}%
\special{pa 3590 3037}%
\special{pa 3590 3076}%
\special{fp}%
\special{pa 3590 3084}%
\special{pa 3590 3122}%
\special{fp}%
\special{pa 3590 3131}%
\special{pa 3590 3169}%
\special{fp}%
\special{pa 3590 3178}%
\special{pa 3590 3216}%
\special{fp}%
\special{pa 3590 3224}%
\special{pa 3590 3263}%
\special{fp}%
\special{pa 3590 3271}%
\special{pa 3590 3310}%
\special{fp}%
\special{pa 3590 3318}%
\special{pa 3590 3356}%
\special{fp}%
\special{pa 3590 3365}%
\special{pa 3590 3403}%
\special{fp}%
\special{pa 3590 3412}%
\special{pa 3590 3450}%
\special{fp}%

%
\special{pn 8}%
\special{ar 2190 1925 737 737 5.7587431 0.5085612}%
%
\special{pn 8}%
\special{ar 4855 1915 2240 2240 2.7095063 3.5612159}%
\put(29.5500,-19.8000){\makebox(0,0)[lb]{$a$}}%
\put(25.6500,-13.6000){\makebox(0,0)[lb]{$b$}}%
\put(36.3500,-13.3000){\makebox(0,0)[lb]{$w$}}%
\put(28.9000,-24.1000){\makebox(0,0)[lb]{$1$}}%
\put(36.2000,-24.1000){\makebox(0,0)[lb]{$r^2$}}%
\end{picture}}%

\end{center}
\vspace{3mm}
\end{rem}

\begin{proof}
Set
$$
A:=\left\{ w\in {\mathcal W} : \underline{w}(r^{-1}, r) >ar^{-\alpha },~ \overline{w}(1, r) - \underline{w}(0,1) <br^{-\alpha} \right\}.
$$
Then we see
\begin{align*}
{\cal Q}(A)
&= {\cal Q}\left(\underline{w}(1, r^2) >a,~ \overline{w}(r, r^2) - \underline{w}(0,r) <b\right) \\
&\geq {\cal Q}\left(\underline{w}(1, r^2) >a,~ \overline{w}(1, r^2) - \underline{w}(0,1) <b\right) > 0
\end{align*}
in view of (\ref{eq:selfsimilar}). 
In view of Lemma \ref{prop:mixing}, there exists ${\mathcal N} \in {\mathscr B}({\mathcal W})$ such that ${\mathcal Q}({\mathcal N})=0$, and for $w\in {\mathcal W}\setminus {\mathcal N}$, $w(0)=0$ and $\{n\in {\mathbb N} : T^n w \in A\}$ is an infinite set.
For $w\in {\mathcal W}\setminus {\mathcal N}$, let $\{ n_{\ell}^w : \ell \in {\mathbb N}\}$ be a strictly increasing sequence in $\{ n\in {\mathbb N} : T^n w \in A\}$.
Then we have
\begin{equation}\label{eq:propRE01}\left\{ \begin{array}{l}
\underline{w}\left(r^{n_{\ell}^w-1}, r^{n_{\ell}^w}\right)\ge \underline{w}\left(r^{n_{\ell}^w-1}, r^{n_{\ell}^w+1}\right) >ar^{\alpha (n_\ell^w -1)}\\
\overline{w}\left(r^{n_\ell^w}, r^{n_{\ell}^w+1}\right) -\underline{w}\left(0, r^{n_{\ell}^w}\right) <br^{\alpha (n_\ell^w -1)} \\
\underline{w}(0,r^{n_{\ell}}) = \underline{w}(0, r^{n_{\ell}+1}).
\end{array}
\right.\end{equation}
For the last equality above, we used the fact that $w(0)=0$ implies
\begin{align*}
\underline{w}(0, r^{n_{\ell}}) &=  \min \left\{\underline{w}(0, r^{n_{\ell}}) ,ar^{\alpha (n_{\ell} -1)} \right\} 
\leq \min \left\{\underline{w}(0, r^{n_{\ell}}), \underline{w}(r^{n_{\ell}-1}, r^{n_{\ell}+1}) \right\} 
= \underline{w}(0, r^{n_{\ell}+1}).
\end{align*}
On the other hand, by applying the first inequality in (\ref{eq:propRE01}) to Lemma \ref{lem:w3}, it holds that  
\begin{align*}
&\int _{{\mathbb R}^d} \left|u_{n_{\ell (k)}^w}(x)\right|^2 e^{-w(x)}{\rm d}x \leq C_2\,k r^{(d-2)(n_{\ell (k)}^w-1)}\exp \left( -ar^{\alpha (n_{\ell (k)}^w-1)}\right) \\
&\quad \hspace{5cm} \times \left( 1+ \frac{(r^d-1) r^{dn_{\ell (k)}^w}\exp \left( - ar^{\alpha (n_{\ell (k)}^w-1)} \right)}{\exp \left(-\overline{w}(0,1)\right)} \right)
\end{align*}
for $k \in {\mathbb N}$. From this, one can get for sufficiently large $k \in \mathbb{N}$ that
\begin{equation}\label{eq:propRE02}
\int _{{\mathbb R}^d} \left|u_{n_{\ell (k)}^w}(x)\right|^2 e^{-w(x)}{\rm d}x  \leq Ck r^{(d-2)n_{\ell (k)}^w}\exp \left( -ar^{\alpha (n_{\ell (k)}^w-1)}+\overline{w}(0,1)\right),
\end{equation}
where $C$ is a constant depending on $d$, $r$, $\alpha$, $\varphi$ and $a$. Moreover, since
$$
{\mathcal E}^w\left(u_{n_{\ell (k)}^w}, u_{n_{\ell (k)}^w}\right) \leq \frac{1}{k}\int_{\mathbb{R}^d}\left|u_{n_{\ell (k)}^w}(x)\right|^2e^{-w(x)}{\rm d}x, \quad k\in {\mathbb N},
$$
we have by (\ref{lem:w1-2}) and (\ref{lem:w2}) that for $k\in {\mathbb N}$
\begin{align}\label{eq:r2}
C_2r^{-2n_{\ell (k)}^w} \exp \left(-\overline{w}\left(r^{n^w_{\ell (k)}}, r^{n^w_{n_{\ell (x)}}\!\!+1}\right)\right) \leq \frac{2}{k} C_1 \exp \left(-\underline{w}\left(0, r^{n^w_{n_{\ell (x)}}\!\!+1}\right) \right).
\end{align}
Then, by virtue of the second and third relations in \eqref{eq:propRE01}, and \eqref{eq:r2}
\begin{align*}
&r^{-2n_{\ell (k)}^w} \exp \left(- br^{\alpha (n_{\ell (k)}^w-1)} \right) \le r^{-2n_{\ell (k)}^w} \exp \left(\underline{w}\left(0, r^{n^w_{n_{\ell (x)}}}\right) - \overline{w}\left(r^{n^w_{\ell (k)}}, r^{n^w_{n_{\ell (x)}}\!\!+1}\right) \right) \\
&= r^{-2n_{\ell (k)}^w} \exp \left(\underline{w}\left(0, r^{n^w_{n_{\ell (x)}}\!\!+1}\right) - \overline{w}\left(r^{n^w_{\ell (k)}}, r^{n^w_{n_{\ell (x)}}\!\!+1}\right)\right) \le \frac{2C_1}{kC_2},
\end{align*}
hence, for $\gamma \in [0,1]$ and $k\in {\mathbb N}$
$$
r^{-2(1-\gamma)n_{\ell (k)}^w} \exp \left(- b(1-\gamma)r^{\alpha (n_{\ell (k)}^w-1)} \right) \le \left(\frac{2C_1}{kC_2}\right)^{1-\gamma}.
$$
From this inequality, we see that the right-hand side of (\ref{eq:propRE02}) is dominated by
\begin{align*}
C\left(\frac{2C_1}{C_2}\right)^{1-\gamma} k^{\gamma}r^{(d-2\gamma)n_{\ell (k)}^w}\exp \left((b(1-\gamma) -a)r^{\alpha (n_{\ell (k)}^w-1)}+\overline{w}(0,1)\right).
\end{align*}
The proof is completed.
\end{proof}

Now we prove Theorem \ref{thm:RE1} by applying Propositions \ref{thm:criterion} and \ref{prop:RE}.

\begin{proof}[Proof of Theorem \ref{thm:RE1}]
We note that
\begin{align*}
&{\mathcal Q} \left(\overline{w}(1,r^2) -\underline{w}(0, 1) <2\underline{w}(1,r^2)\right)\\
&\qquad = {\mathcal Q} \left(\bigcup _{\widetilde{a}\in (0,\infty ) \cap {\mathbb Q}} \left\{\underline{w}(1,r^2) >\widetilde{a},\ \overline{w}(1,r^2) -\underline{w}(0, 1) <2\widetilde{a} \right\} \right) \\
&\qquad \leq \sum _{\widetilde{a} \in (0,\infty ) \cap {\mathbb Q}} {\mathcal Q} \left(\underline{w}(1,r^2) >\widetilde{a},\ \overline{w}(1,r^2) -\underline{w}(0, 1) <2\widetilde{a}\right).
\end{align*}
In view of this fact, the assumption \eqref{assump;RE1} implies that there exists $a >0$ such that
\begin{equation}\label{2a-varepsilon}
{\mathcal Q} \left(\underline{w}(1,r^2) >a,\ \overline{w}(1,r^2) -\underline{w}(0, 1) <2a- 4\varepsilon \right) >0
\end{equation}
for a sufficiently small $\varepsilon > 0$.
Thus, by applying Proposition \ref{prop:RE} with $\gamma = 1/2$ and $b=2a-\varepsilon$, we see that 
\begin{equation}\label{L2diff}
k^{-1/2} \int _{{\mathbb R}^d} \left|u_{n_{\ell (k)}^w}(x)\right|^2 e^{-w(x)}{\rm d}x \longrightarrow 0 \quad {\rm as}~~k \to \infty.
\end{equation}
Then, by Corollary \ref{ex:Bm} and (\ref{L2diff}), 
\begin{align*}
\frac{2}{k}\left(\int _{{\mathbb R}^d} \left|u_{n_{\ell (k)}^w}(x)\right|^2 e^{-w(x)}{\rm d}x\right)\left(\int _{{\mathbb R}^1} \left|u_{n(k)}(x)\right|^2{\rm d}x\right) \le \frac{2\widetilde{C}}{k^{1/2}}\int _{{\mathbb R}^d} \left|u_{n_{\ell (k)}^w}(x)\right|^2 e^{-w(x)}{\rm d}x \longrightarrow 0 
\end{align*}
as $k \to \infty$. Moreover, (\ref{2a-varepsilon}) implies ${\mathcal Q}(\underline{w}(1,r)>a)>0$ and therefore, $X_w(t)$ is recurrent for almost all environments in view of \cite[Theorem 2.2]{KTT1}. Hence,  by virtue of Proposition \ref{thm:criterion}, we can conclude that $(X_{w}(t), B(t))$ is recurrent for almost all environments.
\end{proof}

\section{Applications to explicit random environments}\label{sec:App}
In this section, as applications of a random environment appeared in Section \ref{sec:RE}, we consider the recurrence of the product of diffusion processes in semi-selfsimilar Gaussian and L\'evy random environments, and show Theorems \ref{thm:Gauss} and \ref{thm:Levy}, and Corollary \ref{cor:BroxBrown}.

\subsection{Gaussian random environments}\label{subsec:Gauss}

Let ${\mathcal W}$ be the space of locally bounded Borel measurable functions $w$ on ${\mathbb R}^d$, with the topology generated by the uniform convergence on compact sets. 
We define a probability measure ${\mathcal Q}$ on $({\mathcal W}, {\mathscr B}({\mathcal W}))$ by a Gaussian measure, that is, $(w(x_1), w(x_2),\dots w(x_n))$ has an $n$-dimensional Gaussian distribution under ${\mathcal Q}$, where $x_1,x_2,\dots ,x_n \in {\mathbb R}^d$ for $n\in {\mathbb N}$. 
We assume that ${\mathcal Q}(w(0)=0)=1$ and $E^{\cal Q}[w(x)]=0$ for $x\in {\mathbb R}^d$. Here $E^{\cal Q}$ stands for the expectation with respect to ${\cal Q}$. 
Let $K$ be the covariance kernel of ${\mathcal Q}$, that is, $K(x,y) = E^{\cal Q}[w(x)w(y)]$ for $x,y\in {\mathbb R}^d$. Since ${\mathcal Q}$ is a probability measure on $({\mathcal W}, {\mathscr B}({\mathcal W}))$, $K$ is a measurable function on ${\mathbb R}^d\times {\mathbb R}^d$. 

It is well-known that the law of a Gaussian measure is determined by the mean and the covariance kernel. First, we are going to consider a sufficient condition for \eqref{assump;RE1} in Theorem \ref{thm:RE1}.

\begin{lem}\label{lem:Gauss1}
Assume that for $r >1$
\[
\sup _{1 \le |x| \le r^2} \int _{|y|\leq r^2} K(x,y) {\rm d}y -\inf _{|x|\le 1} \int _{|y|\leq r^2} K(x,y) {\rm d}y <2\inf _{1\le |x|\le r^2} \int _{|y|\leq r^2} K(x,y) {\rm d}y.
\]
Then the assumption \eqref{assump;RE1} holds.
\end{lem}

\begin{proof}
The proof is similar to that of \cite[Lemma 3.1]{KTT1}.
So, we omit the detail and see only the sketch of the proof.
By the general theory of the Gaussian system, for $f\in L^2({\mathbb R}^d, {\rm d}x)$ with compact support, $\int _{{\mathbb R}^d}K(\cdot ,y) f(y) {\rm d}y$ 
is in the Cameron-Martin space ${\mathcal H}$ associated to ${\mathcal Q}$ on $({\mathcal W}, {\mathscr B}({\mathcal W}))$.
In particular,
\begin{equation}\label{eq:CM1}
\int _{|y|\leq r^2} K(\cdot ,y)\, {\rm d}y \in {\mathcal H} .
\end{equation}
On the other hand, since ${\mathcal H}$ is dense in the support of ${\mathcal Q}$, for any $g\in {\mathcal H}$
\[
{\mathcal Q}\left(\sup _{x\in {\mathbb R}^d}|w(x)-g(x)| < \varepsilon \right) >0 \quad {\rm for~any}~~\varepsilon >0.
\]
This inequality and (\ref{eq:CM1}) imply
\begin{equation}\label{eq:CM2}
{\mathcal Q}\left(\sup _{x\in {\mathbb R}^d}\left| w(x)-\int _{|y|\leq r^2} K(x,y)\,{\rm d}y\right| < \varepsilon \right) >0.
\end{equation}
Let
\begin{align*}
\delta := 2\inf _{1\le |x|\le r^2} \int _{|y|\leq r^2} K(x,y) {\rm d}y - \left( \sup _{1\le |x| \le r^2} \int _{|y|\leq r^2} K(x,y) {\rm d}y -\inf _{|x|\le 1} \int _{|y|\leq r^2} K(x,y) {\rm d}y \right)
\end{align*}
and choose $\varepsilon \in (0, \delta /4)$.
Then, if
\[
\sup _{x\in {\mathbb R}^d}\left| w(x)-\int _{|y|\leq r^2} K(x,y) {\rm d}y\right| < \varepsilon ,
\]
we have
\[
\sup _{1\le |x| \le r^2} w(x) -\inf _{|x|\le 1}w(x) < 2\inf _{1\le |x|\le r^2} w(x) + 4\varepsilon -\delta < 2\inf _{1\le |x|\le r^2} w(x) .
\]
Therefore, by (\ref{eq:CM2}) we have the assertion. 
\end{proof}

Next we consider a sufficient condition for the mixing condition.
In the sequel, let $r>1$, $\alpha >0$ and  $T$ be a mapping from ${\mathcal W}$ to ${\mathcal W}$ defined by $Tw(x) = r^{-\alpha} w(rx)$ for $x\in {\mathbb R}^d$ satisfying (\ref{eq:selfsimilar}).
We say that $T$ is strongly mixing if
$$
\lim _{n\rightarrow \infty} {\mathcal Q}(T^n A_1\cap A_2)= {\mathcal Q}(A_1) {\mathcal Q}(A_2), \quad A_1,A_2\in {\mathscr B}({\mathcal W}).
$$
It is well known that every strongly mixing transformation is weakly mixing, hence is ergodic (see \cite{Walters}). Set $D_1:=\{ x\in {\mathbb R}^d; 1<|x|<r\}$. Then, we see the following.

\begin{lem}[Lemma 3.3 in \cite{KTT1}]\label{lem:Gauss2}
If
\[
\lim _{n\rightarrow \infty} r^{-\alpha n} \sup _{x,y\in D_1} K(r^n x,y) =0,
\]
then $T$ is strongly mixing.
\end{lem}

Now we prove Theorem \ref{thm:Gauss} and Corollary \ref{cor:BroxBrown}.

\begin{proof}[Proof of Theorem \ref{thm:Gauss}]
For $w \in {\mathcal W}$, let ${\bf X}^w=(X_w(t), P_x^w)$ be the $d$-dimensional diffusion process associated to $({\mathcal E}^w, {\mathcal F}^w)$ defined by (\ref{eq:DFw}). Note that the condition \eqref{assump;RE1} implies ${\mathcal Q}(\underline{w}(1,r) >a_0) >0$ for some $a_0>0$. Thus, if $T$ on ${\mathcal W}$ is weakly mixing, ${\bf X}^w$ is recurrent for almost all environment (\cite[Theorem 2.2]{KTT1}). 
From this fact with Theorem \ref{thm:RE1}, Lemmas \ref{lem:Gauss1} and \ref{lem:Gauss2}, we have Theorem \ref{thm:Gauss}.
\end{proof}

%

\begin{proof}[Proof of Corollary \ref{cor:BroxBrown}]
Let $w$ be the two-sided Brownian motion on ${\mathbb R}$ under ${\mathcal Q}$. In this case, $w$ and $r^{-1/2} w(r\,\cdot )$ have the same law.
Furthermore, the covariance kernel $K(x,y)$ is given by
\[
K(x,y) = (|x|\wedge |y|) {\bf 1}_{(0,\infty)} (xy), \quad x,y\in {\mathbb R}.
\]
Choose $r$ satisfying $1 < r < \sqrt{2+\sqrt{2}}$. Then, since
\begin{align*}
&\sup _{1 \le |x| \le r^2} \int _{|y|\leq r^2} K(x,y) {\rm d}y = \frac 12 r^4,\quad \inf _{|x|\le 1} \int _{|y|\leq r^2} K(x,y) {\rm d}y =0, \\ 
&\inf _{1\le |x|\le r^2} \int _{|y|\leq r^2} K(x,y) {\rm d}y =r^2-\frac{1}{2},
\end{align*}
(\ref{eq:thmGauss1}) is satisfied. Moreover, it is easy to see that (\ref{eq:thmGauss2}) is satisfied. On the other hand, we note that the Brox's diffusion process $Y_w(t)$ is a time changed process of the one-dimensional diffusion process $X_w(t)$ by the positive continuous additive functional $\int_0^{\cdot}e^{-w(X_w(s))}{\rm d}s$. Then the Dirichlet form $(\check{\mathcal E}^w, \check{\mathcal F}^w)$ corresponding to $Y_w(t)$ is given by (\ref{eq:DFw}) replacing the underlying measure $e^{-w(x)}{\rm d}x$ with $e^{-2w(x)}{\rm d}x$. It is well-known that $X_w(t)$ is recurrent (see \cite{Tanaka93}).  
Hence, by Corollary \ref{maincor} and Theorem \ref{thm:Gauss}, we obtain the assertion.
\end{proof}

\subsection{Products of environments generated by L\'evy processes}\label{subsec:Levy}

Let ${\mathcal W}$ be the space of functions $w$ on ${\mathbb R}$ satisfying the following: $w(0)=0$, $w$ is right (resp. left) continuous with left (resp. right) limits on $[0,\infty)$ (resp. $(-\infty, 0)$).
For $i=1,2,\dots , d$ we set a probability measure ${\mathcal Q}_i$ on $({\mathcal W}, {\mathscr B}({\mathcal W}))$ such that $(w(x), x \in [0,\infty), {\mathcal Q}_i)$ and $(w(-x), x \in  [0,\infty ), {\mathcal Q}_i)$ are independent semi-selfsimilar L\'evy processes with an exponent $\alpha _i \in (0,2]$ (see Definitions 13.4 and 13.12 in \cite{Sa}). 
Define the probability measure ${\mathcal Q}$ on $({\mathcal W}^d, {\mathscr B}({\mathcal W})^{\otimes d})$ by
\[
{\mathcal Q} := \bigotimes _{i=1}^d {\mathcal Q_i} .
\]
Denote the $i$th component of $w \in {\mathcal W}^d$ by $w^i$ and denote $w_+^i (t) := w^i(t)$ and $w_-^i (t) := w^i(-t)$ for $t\in [0,\infty )$.

By a similar argument to the proof of Proposition 2.1 in \cite{KTT2}, we have the following.

\begin{lem}\label{prop:Levy}
Let $i=1,2,\dots ,d$.
If $\alpha _i =2$ or both $w_+^i$ and $w_-^i$ have positive jumps with positive probabilities, then there exists a positive constant $M$ such that for any $a >0$
$$
{\mathcal Q}\left(\left\{\underline{w}^i(0,1) > - M \right\} \cap \left\{\left(w^i(-1) \wedge w^i(1)\right)> a\right\} \right) >0 .
$$
\end{lem}

\begin{proof}
First we prove that there exists $M >0$ such that for any $a >0$
\begin{equation}\label{eq:propLevy01}
{\mathcal Q}\left( \left\{ \inf _{t\in [0,1]} w_+^i(t) > - M \right\} \cap \left\{ w_+^i(1) > a\right\} \right) >0, \quad i=1,2,\dots, d. 
\end{equation}
If $\alpha _i =2$, then $w_+^i$ is a Brownian motion and hence (\ref{eq:propLevy01}) holds.
Assume that $\alpha _i \in (0,2)$.
In this case, we note that the L\'evy measure $\nu$ of $w_+^i$ is not trivial and its Gaussian part is to be $0$.
Since $w_+^i$ has positive jumps with a positive probability, we can choose $\varepsilon \in (0,1]$ such that $\nu ((\varepsilon , \infty)) >0$.
For $i=1,2,\dots, d$, let $v_1^i$, $v^i_2$ and $v^i_3$ be independent L\'evy processes associated to the triplets $(0, \nu ( \cdot \cap (\varepsilon,\infty )),0)$, $(0, \nu (\cdot \cap [-1,\varepsilon ]) ,0)$ and $(0, \nu (\cdot \cap (-\infty ,-1)) ,0)$, respectively.
Then, the equality in law
\begin{equation}\label{eq:propLevy11}
\left(w_+^i(t), t\in [0,\infty )\right) \stackrel{\rm law}{=} \left(v_1^i(t) + v_2^i(t) + v_3^i(t) + ct, t\in [0,\infty )\right)
\end{equation}
holds for a constant $c \in {\mathbb R}$.
Note that $v_2^i$ is right-continuous with left limits almost surely. So there exists $M >0$ such that
\begin{equation}\label{eq:propLevy12}
{\mathcal Q}\left( \sup _{t\in [0,1]} \left|v_2^i(t)\right| < M - |c| \right) >0.
\end{equation}
Also, since $\nu ((\varepsilon ,\infty ))>0$, $v_1^i \geq 0$ almost surely and we have
\begin{equation}\label{eq:propLevy14}
{\mathcal Q}\left(v_1^i(1)> a +M + |c| \right) >0
\end{equation}
for any $a >0$. On the other hand, the definition of the L\'evy measure implies $\nu ((-\infty ,-1)) < \infty$, and hence
\begin{equation}\label{eq:propLevy13}
{\mathcal Q}\left(v_3^i(t) =0 \ \mbox{for}\ t\in [0,1] \right) >0.
\end{equation}
From (\ref{eq:propLevy11}), (\ref{eq:propLevy12}), (\ref{eq:propLevy13}), (\ref{eq:propLevy14}) and the independence of $v_1^i$, $v_2^i$ and $v_3^i$, we then obtain
\begin{align*}
&{\mathcal Q}\left( \left\{\inf _{t\in [0,1]} w_+^i(t) > - M \right\} \cap \left\{ w_+^i(1) > a\right\} \right) \\
&\quad \geq {\mathcal Q}\left( \{ v_1^i(1)> a +M + |c| \} \cap \left\{ \inf _{t\in [0,1]} v_2^i(t) > -M + |c| \right\} \cap \left\{ v_3^i(t) =0 \ \mbox{for}\ t\in [0,1]\right\} \right) \\
&\quad = {\mathcal Q}\left( v_1^i(1)> a +M + |c| \right) {\mathcal Q}\left( \inf _{t\in [0,1]} v_2^i(t) > -M + |c| \right) {\mathcal Q}\left(  v_3^i(t) =0 \ \mbox{for}\ t\in [0,1] \right) > 0.
\end{align*}
Thus, we obtain (\ref{eq:propLevy01}). Similarly to above, we also have
\begin{equation}\label{eq:propLevy02}
{\mathcal Q}\left( \left\{ \inf _{t\in [0,1]} w_-^i(t) > - M \right\} \cap \left\{ w_-^i(1) > a\right\} \right) >0.
\end{equation}
Now, on account of (\ref{eq:propLevy01}) and (\ref{eq:propLevy02})
\begin{align*}
&{\mathcal Q}\left( \left\{ \underline{w}^i(0,1) > - M \right\} \cap \left\{ w^i(-1) \wedge w^i(1) > a\right\} \right) \\
&\quad = {\mathcal Q}\left( \left\{ \inf _{t\in [0,1]} w_+^i(x) > - M \right\} \cap \left\{ w_+^i(1) > a\right\} \right)  {\mathcal Q}\left( \left\{ \inf _{t\in [0,1]} w_-^i(t) > - M \right\} \cap \left\{ w_-^i(1) > a\right\} \right) >0
\end{align*}
for any $a>0$ and $i=1,2,\dots, d$.
\end{proof}

From this lemma, we have the following.

\begin{lem}\label{lem:Levy}
Let $i=1,2,\dots ,d$. If $\alpha _i =2$ or both $w_+^i$ and $w_-^i$ have positive jumps with positive probabilities, then there exists $\widehat{a} >0$ such that for any $\varepsilon >0$
\[
{\mathcal Q} \left( \underline{w}^i(1,r^2) >\widehat{a},\ \overline{w}^i(1,r^2) -\underline{w}^i(0,1) < \widehat{a}(1+\varepsilon ) \right) >0.
\]
\end{lem}

\begin{proof}
It is sufficient to show that there exists $\widehat{a}>0$ such that for any $\varepsilon >0$
\begin{align}
{\mathcal Q} \left( \inf _{t \in [1,r^2]} w_+^i(t) >\widehat{a},\ \sup _{t \in [1,r^2]} w_+^i(x) -\inf _{t\in [0,1]} w_+^i(x) < \widehat{a}(1+\varepsilon )  \right) >0, \quad i=1,2,\dots, d , \label{eq:lemLevy01}
\end{align}
because the proof of (\ref{eq:lemLevy01}) for $(w^i_-(x), x \in (-\infty, 0))$ is almost same.
Since $w_+^i$ is a L\'evy process, $(w_+^i(t)- w_+^i(1), t\in [1,\infty ))$ and $(w_+^i(t), t\in [0,1])$ are independent.
Then, for any $\widehat{a} >0$ and $i=1,2,\dots, d$, we have
\begin{align*}
&{\mathcal Q} \left( \inf _{t \in [1,r^2]} w^i_+(t) >\widehat{a},\ \sup _{t \in [1,r^2]} w^i_+(t) -\inf _{t \in [0,1]} w^i_+(t) < \widehat{a}(1+\varepsilon ) \right)\\
&\quad = {\mathcal Q} \left( \widehat{a}- w^i_+(1)< w^i_+(t) - w^i_+(1) <\widehat{a}(1+\varepsilon ) - w^i_+(1) + \inf _{t\in [0,1]} w^i_+(t)\quad \mbox{for}\ t \in [1,r^2] \right) \\
&\quad \geq {\mathcal Q} \left( \left\{ -\frac{\varepsilon \widehat{a}}{4} < w^i_+(t) - w^i_+(1) <\frac{\varepsilon \widehat{a}}{4} \quad \mbox{for}\ t \in [1,r^2] \right\} \phantom{\left\{ \inf _{t\in [0,1]} w^i_+(t) \right\}} \right. \\
&\quad \hspace{5cm} \left. \cap \left\{ \inf _{t\in [0,1]} w^i_+(t) > - \frac{\varepsilon \widehat{a}}{4} \right\} \cap \left\{\frac{(4+\varepsilon)\widehat{a}}{4}< w^i_+(1)<\frac{(2+\varepsilon )\widehat{a}}{2}\right\} \right) \\
&\quad = {\mathcal Q} \left( -\frac{\varepsilon \widehat{a}}{4} < w^i_+(t) - w^i_+(1) <\frac{\varepsilon \widehat{a}}{4} \quad \mbox{for}\ t \in [1,r^2] \right) \\
&\quad \hspace{4cm} \times {\mathcal Q} \left( \left\{ \inf _{t\in [0,1]} w^i_+(t) > - \frac{\varepsilon \widehat{a}}{4} \right\} \cap \left\{\frac{(4+\varepsilon)\widehat{a}}{4}< w^i_+(1)<\frac{(2+\varepsilon )\widehat{a}}{2}\right\} \right) \\
&\quad \geq {\mathcal Q} \left( \sup _{t\in [1,r^2]} |w^i_+(t) - w^i_+(1)| < \frac{\varepsilon \widehat{a}}{4}\right) \\
&\quad \hspace{4cm} \times {\mathcal Q} \left( \left\{ \inf _{t\in [0,1]} w^i_+(t) > - \frac{\varepsilon \widehat{a}}{4} \right\} \cap \left\{ \frac{(4+\varepsilon)\widehat{a}}{4}< w^i_+(1)<\frac{(2+\varepsilon )\widehat{a}}{2}\right\} \right) .
\end{align*}
Let $M$ be the constant appeared in Lemma \ref{prop:Levy} and let $\widetilde{M} >0$ be a constant satisfying
\[
{\mathcal Q} \left( \sup _{t\in [1,r^2]} \left|w^i_+(t) - w^i_+(1)\right| <\widetilde{M} \right) >0 .
\]
Then, by taking $\widehat{a} \in [4(M \vee \widetilde{M}) /\varepsilon, \infty)$, we have
\begin{equation}\label{ahat1}
{\mathcal Q} \left( \sup _{t\in [1,r^2]}\left|w^i_+(t) - w^i_+(1)\right| < \frac{\varepsilon \widehat{a}}{4}\right) >0 .
\end{equation}
On the other hand, by noting the fact that \eqref{eq:propLevy01} holds for any $a>0$, we can take $\widehat{a} >0$ such that
\begin{align}\label{ahat2}
&{\mathcal Q} \left( \left\{ \inf _{t\in [0,1]} w^i_+(t) > - \frac{\varepsilon\widehat{a}}{4} \right\} \cap \left\{\frac{(4+\varepsilon)\widehat{a}}{4}< w^i_+(1)<\frac{(2+\varepsilon )\widehat{a}}{2}\right\} \right) \nonumber \\
& \qquad \ge {\mathcal Q} \left( \left\{ \inf _{t\in [0,1]} w^i_+(t) > - M \right\} \cap \left\{\frac{(4+\varepsilon)\widehat{a}}{4}< w^i_+(1)<\frac{(2+\varepsilon )\widehat{a}}{2}\right\} \right) >0 .
\end{align}
Therefore, we can conclude the assertion for $\widehat{a}>0$ satisfying both \eqref{ahat1} and \eqref{ahat2}.
\end{proof}

Now we prove Theorem \ref{thm:Levy} by applying Propositions \ref{thm:criterion} and \ref{prop:RE}, and Lemma \ref{lem:Levy}.

\begin{proof}[Proof of Theorem \ref{thm:Levy}]
Let define a random function $w$ by
\[
w(x):= \sum _{i=1}^d w^i \left( x^{(i)}\right) , \quad x=(x^{(1)}, x^{(2)},\dots , x^{(d)}) \in {\mathbb R}^d.
\]
For this $w$, let ${\bf X}^{w}= (X_{w}(t), P_{x}^{w})$ be the diffusion process associated to the Dirichlet form $({\mathcal E}^w, {\mathcal F}^w)$ given by (\ref{eq:DFw}). Then, it is the $d$-dimensional direct products of diffusion processes in products of random environments generated by one-dimensional semi-selfsimilar L\'evy processes $\{(w^i(x), x \in \mathbb{R}), i=1,2,\dots, d\}$, that is,
$$
X_w(t)=\left(X_{w^1}^1(t), X_{w^2}^2(t), \dots, X_{w^d}^d(t)\right).
$$
We remark that the components of $X^{(w)}_t$ are independent for each environment $w$, because
\[
e^{-w(x)} {\rm d}x = \prod _{i=1}^d e^{-w^i(x^{(i)})} {\rm d}x^{(i)}
\]
for $x= (x^{(1)}, x^{(2)}, \dots , x^{(d)}) \in {\mathbb R}^d$.
In view of Propositions \ref{thm:criterion} and \ref{prop:RE}, and Lemma \ref{lem:Levy} we obtain the assertion.
\end{proof}

\section*{Acknowledgment}
The authors would like to thank the referee for his/her valuable comments and suggestions.
This work was supported by JSPS KAKENHI Grant number 17K05304 and 17K14204.

\bibliographystyle{plain}
\bibliography{RandomMedia.bib}

\begin{thebibliography}{99}
\bibitem{Brox}{Th.~Brox},
         \emph{A one-dimensional diffusion process in a Wiener medium}, Ann. Probab. {\bf 14}, no. 4~(1986), 1206--1218.

\bibitem{Devulderetal}{A.~Devulder, N.~Gantert and F.~P\`ene}
         \emph{Collisions of several walkers in recurrent random environments}, Electron. J. Probab. {\bf 23}, no. 90~(2018), 1--34.
         
\bibitem{FNT}{M.~Fukushima, S.~Nakao and M.~Takeda}, 
         \emph{On Dirichlet form with random data - recurrence and homogenization}, Stochastic Processes - Mathematics and Physics II (Bielefeld, 1985), (eds. S.~Albeverio, Ph. Blanchard and L.~Streit), Lect. Notes in Math., {\bf 1250}, Springer-Verlag (1987), 87--97.

\bibitem{FuO}{M. Fukushima and Y. Oshima},
\emph{On the skew product of symmetric diffusion processes}, Forum Math., {\bf 1}, no. 2~ (1989), 103--142.


\bibitem{FOT}
    {M.~Fukushima, Y.~Oshima and M.~Takeda},
   \emph{Dirichlet forms and symmetric Markov processes}, Second revised and extended edition. de Gruyter Studies in Mathematics, {\bf 19}. Walter de Gruyter \& Co., Berlin, 2011.
   
\bibitem{ICHI}
    {K.~Ichihara},
    \emph{Some global properties of symmetric diffusion processes}, Publ. Res. Inst. Math. Sci. {\bf 14}, no.2~ (1978), 441--486.
    
\bibitem{IW}
    {N.~Ikeda and S.~Watanabe},
    \emph{Stochastic differential equations and diffusion processes}, 2nd edn. North Holland, Amsterdam (1989).  
    
\bibitem{KIM}
    {D.~Kim},
    \emph{Some limit theorems related to multi-dimensional diffusions in a random environment}, J. Korean Math. Soc. {\bf 48}, no. 1~ (2011), 147--158. 

\bibitem{KTT1}{S. Kusuoka, H. Takahashi and Y. Tamura},
             \emph{Recurrence of the Brownian motion in multidimensional semi-selfsimilar environments and Gaussian
              environments},
             Potential Anal. {\bf 43} (2015), 695--705.

\bibitem{KTT2}{S. Kusuoka, H. Takahashi and Y. Tamura},
             \emph{Topics on multi-dimensional Brox's diffusions}, In harmonic analysis and nonlinear partial differential equations, RIMS K\^oky\^uroku Bessatsu, B59, 31--44. Res. Inst. Math. Sci. (RIMS), Kyoto, 2016.

\bibitem{KTT3}{S. Kusuoka, H. Takahashi and Y. Tamura},
             \emph{Recurrence and transience properties of multi-dimensional diffusion processes in selfsimilar and semi-selfsimilar random environments}, Electron. Commun. Probab., {\bf 22}, no. 4~ (2017), 1--11.
             
\bibitem{Mathieu}{P.~Mathieu},
             \emph{Zero white noise limit through Dirichlet forms, with application to diffusions in a random environment}, Probab. Th. Related Fields, {\bf 99}~(1994), 549--580.

\bibitem{Ok}{H. Okura},
             \emph{Recurrence criteria for skew products of symmetric Markov processes}, Forum Math., {\bf 1}, no. 4~ (1989), 331--357.

\bibitem{Sa}{K. Sato},
             \emph{L\'evy processes and infinitely divisible distributions}, Cambridge Studies in Advanced Mathematics, {\bf 68}, Cambridge University Press, Cambridge, 2013.
           
\bibitem{Sil}{M.~L.~Silverstein},
             \emph{Symmetric Markov Processes}, Lect. Note in Math. {\bf 426}, Springer Berlin (1974).
             
\bibitem{Sinai}{Y.~G.~Sinai},
             \emph{The limiting behavior of a one-dimensional random walk in a random medium}, Theor. Probab. Appl., {\bf 27}~ (1982), 256--268.
             
             
\bibitem{Takahashi04}{H.~Takahashi},
             \emph{Recurrence and transience of multi-dimensional diffusion processes in reflected Brownian environments}, Statist. Probab. Lett. {\bf 69}, no. 2~(2004), 171--174.
             
\bibitem{TakaTamu}{H.~Takahashi and Y.~Tamura},
             \emph{Recurrence of multi-dimensional diffusion processes in Brownian environments}, Discrete Contin. Dyn. Syst. 2015, Dynamical systems, differential equations and applications. 10th AIMS Conference. Suppl., 1034--1040.
             
\bibitem{Tanaka93}{H.~Tanaka},
             \emph{Recurrence of a diffusion process in a multidimensional Brownian environment}, Proc. Japan Acad. Ser. A Math. Sci. {\bf 69}, no. 9~ (1993), 377--381.
            
\bibitem{Walters}{P.~Walters},
             \emph{An introduction to ergodic theory}, graduate texts in mathematics, {\bf 79}. Springer, New York (1982).
\end{thebibliography}

\vskip 0.4cm
\noindent
Daehong~Kim \\
Faculty of Advanced Science and Technology, Kumamoto University, Japan, \\
{\tt daehong@gpo.kumamoto-u.ac.jp}. 
\vskip 0.3cm
\noindent
Seiichiro Kusuoka \\
Department of Mathematics, Graduate School of Science, Kyoto University, Japan, \\
{\tt kusuoka@math.kyoto-u.ac.jp}.

\end{document}